\documentclass[10pt,reqno]{amsart}
\usepackage[T1]{fontenc}
\usepackage{lmodern}
\usepackage{microtype}
\RequirePackage{eucal,fullpage,amsmath,amsthm,amssymb,mathrsfs,stmaryrd,color,enumerate,accents}
\usepackage[all]{xy}
\usepackage{url}
\numberwithin{equation}{section}

\usepackage{mymacros}
\usepackage{quiver}
\usepackage{comment}

\usepackage{tikz-cd}

\usepackage{relsize}
\usepackage[bbgreekl]{mathbbol}
\usepackage{amsfonts}
\usepackage[dvipsnames]{xcolor} % gives OliveGreen, etc.
\usepackage[pdfusetitle, pagebackref, hyperindex]{hyperref}

\newcommand{\cY}{{\mathcal{Y}}}
\newcommand{\cO}{{\mathcal{O}}}

\newcommand{\cA}{{\mathcal{A}}}
\newcommand{\cB}{{\mathcal{B}}}

\newcommand{\cC}{\mathcal{C}}

\newcommand{\cR}{\mathcal{R}}
\newcommand{\cT}{\mathcal{T}}
\newcommand{\calS}{\mathcal{S}}

\newcommand{\G}{\mathbf{G}}
\newcommand{\ZZ}{\mathbf{Z}}
\newcommand{\N}{\mathbf{N}}

\newcommand{\FF}{\mathbf{F}}
\newcommand{\QQ}{\mathbf{Q}}

\newcommand{\spec}{{\mathrm{Spec}}}

\newcommand{\Maps}{\mathrm{Maps}}

\newcommand{\Shv}{\mathrm{Shv}}

\newcommand{\Qcoh}{\mathrm{QCoh}}

\newcommand{\Gr}{\mathrm{Gr}}

\newcommand{\Pic}{\mathrm{Pic}}

\newcommand{\fil}{\mathrm{Fil}}

\newcommand{\gr}{\mathrm{gr}}

\newcommand{\into}{\hookrightarrow}

\newcommand{\crys}{\mathrm{crys}}

\newcommand{\conj}{\mathrm{conj}}

\newcommand{\Cone}{\mathrm{Cone}}

\newcommand{\tensor}{\otimes}

\newcommand{\cofib}{\mathrm{cofib}}

\newcommand{\cn}{\mathrm{cn}}

\DeclareSymbolFontAlphabet{\mathbb}{AMSb} %to ensure that the meaning of \mathbb does not change
\DeclareSymbolFontAlphabet{\mathbbl}{bbold}
\newcommand{\Prism}{{\mathlarger{\mathbbl{\Delta}}}}

\newcommand{\cosimp}[3]{\xymatrix@1{#1 \ar@<.4ex>[r] \ar@<-.4ex>[r] & {\ }#2 \ar@<0.8ex>[r] \ar[r] \ar@<-.8ex>[r] & {\ } #3 \ar@<1.2ex>[r] \ar@<.4ex>[r] \ar@<-.4ex>[r] \ar@<-1.2ex>[r] & \cdots }}
\newcommand{\simp}[3]{%
  \xymatrix@1{%
    \cdots \ar@<1.2ex>[r] \ar@<.4ex>[r] \ar@<-.4ex>[r] \ar@<-1.2ex>[r] &
    {\ }#3 \ar@<0.8ex>[r] \ar[r] \ar@<-.8ex>[r] &
    {\ }#2 \ar@<.4ex>[r] \ar@<-.4ex>[r] &
    #1
  }%
}
\newcommand{\simpII}[2]{%
  \xymatrix@1{%
    \cdots \ar@<0.8ex>[r] \ar[r] \ar@<-0.8ex>[r] &
    {\ }#2 \ar@<0.4ex>[r] \ar@<-0.4ex>[r] &
    #1
  }%
}

\begin{document}
\newtheorem{theorem}{Theorem}[section]
\newtheorem*{theorem*}{Theorem}
\newtheorem*{condition*}{Condition}
\newtheorem*{definition*}{Definition}
\newtheorem{proposition}[theorem]{Proposition}
\newtheorem{lemma}[theorem]{Lemma}
\newtheorem{corollary}[theorem]{Corollary}
\newtheorem{claim}[theorem]{Claim}

\theoremstyle{definition}
\newtheorem{definition}[theorem]{Definition}
\newtheorem{question}[theorem]{Question}
\newtheorem{goal}[theorem]{Goal}
\newtheorem{remark}[theorem]{Remark}
\newtheorem{guess}[theorem]{Guess}
\newtheorem{example}[theorem]{Example}
\newtheorem{condition}[theorem]{Condition}
\newtheorem{warning}[theorem]{Warning}
\newtheorem{caution}[theorem]{Caution}
\newtheorem{notation}[theorem]{Notation}
\newtheorem{construction}[theorem]{Construction}
\newtheorem{recollection}[theorem]{Recollection}
\newtheorem{variant}[theorem]{Variant}

\title{De Rham affineness of the Nygaard filtered prismatization in positive characteristic}
\begin{abstract}
Let $k$ be a perfect ring of characteristic $p>0$, and let $R$ be an animated $k$-algebra. This note aims to show that the Nygaard filtered prismatization $R^\nyg$ of $R$ is naturally isomorphic, as a stack over $k^\nyg$, to the relative spectrum over $k^\nyg$ of the Rees algebra of the Nygaard filtered prismatic cohomology of $R$ relative to $W(k)$. In doing so, we axiomatise the functorial affineness property displayed by the relative Nygaard filtered prismatization,  and dub it \emph{de Rham affineness} after the fundamental example of the functor sending an animated ring to its relative de Rham stack. While we treat this concept as an organising tool for the author's forthcoming work studying the syntomification of Frobenius liftable schemes \cite{Sah25Syn}, we are able to frame some questions based on a structural result of independent interest: a functor to stacks which is de Rham affine often arises via ring stacks through transmutation.
\end{abstract}
\author{Shubhankar Sahai}
\address{Department of Mathematics, University of California San Diego, La Jolla, CA 92093, USA}
\email{ssahai@ucsd.edu}
\subjclass[2020]{Primary $14\text{F}30$; Secondary $14\text{F}40$, $14\text{G}17$,$14\text{G}45$ }
\maketitle
\setcounter{tocdepth}{1}
\tableofcontents
\section{Introduction}

Fix a prime $p$ and a perfect ring $k$ of characteristic $p.$ Our main result is a relative affineness statement for the Nygaard filtered prismatization of animated $k$-algebras. 

We give a somewhat extended introduction, which is organized as follows: \S~\ref{sec: background} provides background contextualizing our theorem; \S~\ref{sec: results intro} summarizes the main result; \S~\ref{sec: overview} outlines the layout of the note; \S~\ref{forthcoming} places the result in the context of our forthcoming work; and \S~\ref{sec: conventions} records conventions. 

We advise the reader to glance at the conventions first, as we freely use the language of $\icats$ and derived algebraic geometry; this choice is reflected in our notation.

\subsection{Background.}\label{sec: background}

Let $X$ be a smooth scheme over $k$. The de Rham cohomology $\dR_{X/k}$ of $X$, arising as the hypercohomology of the de Rham complex on $X$, is a coconnective $E_\infty$-algebra over $k$.

The $E_\infty$-ring $\dR_{X/k}$ can also be recovered via the theory of the de Rham stack. This produces, for smooth $X$, a stack $(X/k)^\dR$ so that as $E_\infty$-algebras over $k$, one has $$\Gamma((X/k)^\dR,\cO_{(X/k)^\dR})\simeq \dR_{X/k}.$$

In \cite{reconstruction_mondal}, Mondal answers a question of Illusie by showing that when $X$ is affine then $\dR_{X/k}$ along with its multiplicative structure is enough to \emph{reconstruct} the stack $(X/k)^\dR$, in fact showing that the stack is the \emph{spectrum} of the coconnective $E_\infty$-ring $\dR_{X/k}$ in the sense of \cite[Construction 1.2]{MM24} i.e.   as stacks on $k$-algebras one has 
$$(X/k)^\dR\simeq \Maps_{\Calg_k}(\dR_{X/k},-)$$ where $\Calg_k$ is the $\icat$ of $E_\infty$-algebras over $k.$

The two viewpoints on de Rham cohomology reflect different structures of the theory, and one can hope to use this result to pass from one side to the other as the next example shows.

\begin{example}[The de Rham gerbe]\label{eg: petrov example}
   The de Rham stack $(X/k)^\dR$ admits a morphism $(X/k)^\dR\to X^{(1)}$, where the target is the Frobenius twist of $X$ relative to $k$. This makes the source a gerbe banded by the flat group scheme $\mathbf{V}(T^{\sharp}_{X^{(1)}/k})$, obtained by the divided power completion at the origin of the tangent space $T_{X^{(1)}/k}$. This gerbe is closely related to the gerbe of Frobenius liftings of $X$\footnote{In fact the gerbe of Frobenius liftings is induced by the de Rham gerbe in a precise sense; for the Hodge-Tate gerbe (a Frobenius `untwist' of the de Rham gerbe), this relation is explained in \cite[Remark 5.13 (3)]{APCII}. See \cite[Appendix, Proposition 1]{MehtaSrinivas1987} for a proof that the induced gerbe (from the de Rham gerbe) is indeed that of Frobenius lifts. It is a remarkable fact that splitting the gerbe of Frobenius liftings also splits the de Rham gerbe. We thank the referee for pointing this out.}  and %therefore 
   splits when $X/k$ admits a $p$-completely flat Frobenius lift to $W(k)$ by \cite[Proposition 5.12]{APCII}. However, viewing the stack as $\spec(\dR_{X/k})$ one should expect that it splits as soon as $X$ admits a flat Frobenius lift mod $p^2$ because $\dR_{X/k}$ splits by \cite{DeligneIllusie1987Relevements}. Indeed, this is true, see \cite[Remark 4.3]{Petrov2025QuasiFSplitDeRham}\footnote{Although we do not know if this was the original motivation for this result.}. For more on this and other deep applications of the gerbe perspective on the de Rham stack, see \cite{Petrov2025QuasiFSplitDeRham}. 
\end{example}

The theory of prismatic cohomology, discovered in \cite{BhattMorrowScholze2018IntegralPadicHodge}, \cite{BhattMorrowScholze2019THHandPadicHodge} and \cite{BS19} can be viewed as a crystalline lift of the theory of de Rham cohomology of mixed characteristic schemes, and has deep applications in several related fields. We refer the reader to \cite{Bhatt2023AlgGeomMixedChar} for a (relatively) recent overview.

In \cite{Drin20}, \cite{APC}, \cite{APCII} and \cite{Bhatt22}, the authors introduced a viewpoint on absolute (and relative) prismatic cohomology which gives a notion of coefficients for the theory by associating to any smooth $p$-adic formal scheme $X/\ZZ_p$ a $p$-adic formal stack $X^\Prism$ and defining prismatic crystals as the stable $\icat$ $\Qcoh(X^\Prism)$. Similarly one has stacks $X^\nyg$ and $X^\syn$ capturing the Nygaard filtration on $X$ and the syntomic cohomology of $X.$ 

We refer the reader to \cite{Drinfeld2019StackyCrystalline} for a gentle introduction to these ideas.

Let $(A,I)$ be a prism and set $\oA=A/I$. In \cite{APCII} the authors introduce relative prismatization which associates to any (possibly derived) $p$-adic formal scheme $X/\oA$ a $(p,I)$-adic formal stack $(X/A)^\Prism$ with the property that $$\Gamma((X/A)^\Prism,\cO)\simeq \Prism_{X/A}$$ in wide enough generality.

In \cite[Theorem 7.17]{APCII} the authors prove the following theorem.

\begin{theorem}[Bhatt-Lurie]\label{eq: rel prism is de rham affine}
    When $X/\oA$ is affine, then there is an equivalence of $(p,I)$-adic formal stacks over $A$
$$
    (X/A)^\Prism\simeq \Maps_{\dalg^{(p,I)\text{-comp}}_{A}}(\Prism_{X/A}, -).$$

\end{theorem}

Here $\dalg_A$ denotes the $\icat$ of Bhatt-Mathew and Raksit's derived $A$-algebras \cite[\S~4]{Rak20} and $\dalg_A^{(p,I)\text{-comp}}$ is the localizing subcategory of $\dalg_A$ consisting of $(p,I)$-complete derived $A$-algebras. The latter category satisfies several pleasing properties and is studied reasonably carefully in \cite[\S~2]{dalggauges}.

We next record an example of the application of the equivalence of Theorem \ref{eq: rel prism is de rham affine}.

\begin{example}
    The authors of \cite{APCII} use the equivalence of Theorem \ref{eq: rel prism is de rham affine} to prove several structural facts about the relative prismatization, showing, for example, that under mild finiteness conditions on the affine derived $p$-adic formal $\oA$-scheme $X$, the stack $(X/A)^\Prism$ is classical \cite[Proposition 7.22 ]{APCII}.
    \end{example}

\begin{remark}
    The reason Mondal's result does not need derived algebras is because he works at the level of classical stacks, implicitly using the fact that the de Rham stack for $\GG_a$ is known to be classical\footnote{This can be deduced independently of the result itself. For example see the argument here \url{https://mathoverflow.net/a/501552}.}. Our viewpoint, for the sake of this note, is that the true objects are \emph{derived}. See also \cite[\S~1.4.7]{Drin20}.
\end{remark}

\begin{remark}[The property of de Rham affineness]\label{remark: de Rham affineness justification}
    The property in Theorem \ref{eq: rel prism is de rham affine} implies that mod any power of $(p,I)$, the stack $(X/A)^\Prism$ is affine in the sense of To\"en \cite{Toen2006ChampsAffines} and Mathew--Mondal \cite{MM24}. On the other hand the isomorphism of \ref{eq: rel prism is de rham affine} is an intrinsic property of the functor of relative prismatization as we now explain. %somewhat stronger than \emph{just being affine}. 
    For any stack $Y$ over $\spf(W(k))$, there is always an adjunction morphism $Y\to \uspf(\Gamma(Y,\cO)).$ However, even when $Y$ is an affine stack, one should not expect this morphism to be an equivalence as the functor $\uspf$ is not necessarily fully faithful unless one restricts to $n$-truncated derived algebras \cite[Remark 3.13]{MM24}\footnote{As mentioned in \emph{loc.cit.}, the graded derived algebra $\QQ[t,t^{-1}]$ with $|t|=2$ has $\uspec(\QQ[t,t^{-1}])=\emptyset.$ Therefore $\uspec$ is not fully faithful.}.

    On the other hand, for the relative prismatization, the proof of \cite[Theorem 7.17]{APCII} shows that the canonical map $$(X/A)^\Prism\to \uspf(\Prism_{X/A}) ,$$ induced from adjunction, is always an isomorphism irrespective of the (co)connectivity of $\Prism_{X/A}$. As this isomorphism is functorial in $X$, this remarkable fact is essentially a property intrinsic to the functor of relative prismatization. In Definition \ref{def: de rham affine} we dub this property \emph{de Rham affineness} and abstract it for our later use.
    
    We refer the reader to our forthcoming work explained in \S~\ref{forthcoming} for the utility of this definition.

\end{remark}

\begin{definition}[De Rham affineness]\label{def: de rham affine}
   Let $A$ be an animated ring and $\aring_A$ be the $\icat$ of animated $A$-algebras. Let $S$ be a stack with an affine diagonal, $\stacks_S$ the $\icat$ of stacks over $S$ and $\dalg(S)$ be the $\icat$ of derived algebras in quasi-coherent sheaves on $S$. Assume that we are given a functor $$\frak{X}\colon \aring^{\op}_A\to \stacks_S.$$ We will say that $\frak{X}$ is \emph{de Rham affine relative to $S$} if there is a functor $$\mathcal{T}\colon \aring_A\to \dalg(S)$$ and a natural in $B\in \aring_A$ isomorphism of $S$-stacks 
   $$\frak{X}(B)\simeq \uspec_{S}(\cT(B))$$ where $\uspec_S(-)$ denotes the relative spectrum of $\cT(B)$ i.e. the functor corepresented by $\cT(B)$  over $S.$ 
   In particular there is an equivalence of $\stacks_S$-valued functors 
   $$\mathfrak{X}(-)\simeq \uspec_S(\cT(-)).$$
\end{definition}

\begin{comment}
\begin{definition}[De Rham affineness]\label{def: de rham affine}
   Let $A$ be an animated ring and $\aring_A$ be the $\icat$ of animated $A$-algebras. Let $S$ be a stack with an affine diagonal, $\stacks_S$ the $\icat$ of stacks over $S$ and $\dalg(S)$ be the $\icat$ of derived algebras in quasi-coherent sheaves on $S$. %Assume that we are given a functor 
   Let $\frak{X}(-)\colon \aring_A\to \stacks_S$ be a functor. We will say that $\frak{X}(-)$ is \emph{de Rham affine relative to $S$} if there is a functor $\cT(-)\colon \aring_A\to \dalg(S)$ and a factorization 
  $$
\begin{tikzcd}
 \aring_A \arrow[r, "\fX(-)"] \arrow[d, "\cT(-)"', dashed] & \stacks_S \\
\dalg(S) \arrow[ru, "\underline{\mathrm{Spec}}_S(-)"']                 &          
\end{tikzcd}
   $$ where $\underline{\mathrm{Spec}}_S(-)$ sends any derived algebra $A\in \dalg(S)$ to the stack $\Maps_{\dalg(S)}(A,-)$ on $S$-schemes.

   %natural in $B\in \aring_A$ isomorphism of $S$-stacks 
  % $$\frak{X}(B)\simeq \uspec_{S}(\cT(B))$$ where $\uspec_S(-)$ denotes the relative spectrum of $\cT(A)$ i.e. the functor corepresented by $\cT(A)$  over $S.$ 
   In particular there is an equivalence of $\stacks_S$-valued functors 
   $$\mathfrak{X}(-)\simeq \uspec_S(\cT(-)).$$
\end{definition}
\end{comment}

\begin{example}\label{eg: de rham stack is de rham affine}
    While Mondal's result in \cite{reconstruction_mondal} shows that the stack $\GG_a^\dR$ is affine in the sense of \cite{Toen2006ChampsAffines}, the result of Bhatt-Lurie, mentioned in Theorem \ref{eq: rel prism is de rham affine}, implies that the functor $(-/k)^\dR$ which sends any animated $k$-algebra $R$ to its de Rham stack $(R/k)^\dR$ is de Rham affine. Indeed, in that setting there is an isomorphism, natural in $R$, of $k$-stacks $$(R/k)^\dR\simeq \uspec(\dR_{R/k}).$$

    More generally their result implies that for any prism $(A,I)$ the relative prismatization functor $(-/\oA)^\Prism$ is de Rham affine. 
\end{example}

\begin{remark}[Colimit preserving de Rham affine functors arise via transmutation]\label{rem: colim preserving de rham affine}
    The notion of de Rham affineness of a stack valued functor is treated as an organizing principle in our forthcoming work. However, we remark that the definition seems robust. A consequence is the following. In the notation of Definition \ref{def: de rham affine} if $\cT$ is a colimit preserving functor then $\uspec_S(\cT(-))$ arises via \emph{transmutation.} This offers one, perhaps anachronistic, reason why the functors of prismatization arise via ring stacks.
    See Corollary \ref{cor: colimit preserve cohomology arises from transmutation}.
\end{remark}

\begin{remark}[To\"en's affineness]
   Note that de Rham affineness is a property of a functor to a category of relative stacks. It implies in particular that such a functor is pointwise affine in the sense of To\"en and Mathew--Mondal \cite{Toen2006ChampsAffines}, \cite{MM24}. 
\end{remark}

\begin{comment}
\begin{example}\label{eg: rel prismatization is de rham affine}
    The result of Theorem \ref{eq: rel prism is de rham affine} implies that the relative prismatization is de Rham affine over $\spf(W(k))$. The functor $\cT$ sends $R\mapsto \Prism_{R/A}$ where the target only depends on the $p$-completion of $R.$
\end{example}

\end{comment}

\begin{remark}
    The notion of de Rham affineness, as abstracted in Definition \ref{def: de rham affine}, is implicit in the proof of \cite[Theorem 7.17]{APCII}. %Therefore we do not take any credit for its discovery. We also hasten to add that Remark \ref{rem: colim preserving de rham affine} is inspired by Mondal \cite{reconstruction_mondal}.
\end{remark}

In \cite{Drin20} and \cite{Bhatt22}, for any animated $p$-complete $\oA$-algebra $R$, certain refinements of the relative prismatization functors $(R/A)^\Prism$ are introduced in the case where $(A,I)$ is a \emph{perfect prism} i.e. $\oA$ is an integral perfectoid ring. 

These associate to any animated $p$-complete $\oA$-algebra $R$, a $(p,I)$-formal stack $(R/A)^\nyg$ with the additional property that the structure morphism $(R/A)^\nyg\to \spf(A)$ factors over $\spf(A)\times \agm$, thus giving $(R/A)^\nyg$ the structure of a $(p,I)$-formal \emph{filtered} stack. 

More precisely, the relative pushforward of the structure sheaf of $(R/A)^\nyg$ along the structure morphism to $\spf(A)\times \agm$ is naturally a filtered $(p,I)$-complete derived algebra under the identification\footnote{For all such folklore identifications, we refer the reader to \cite[\S\S~2-4]{dalggauges}} $$\qcoh(\spf(A)\times \agm)\simeq \fil\Mod_A^{(p,I)\text{-comp}}.$$

The filtration on the pushforward, in many cases, captures the \emph{Nygaard filtration} on relative prismatic cohomology \cite[\S~5.1]{APC}. This filtration can be thought of as a $p$-adic analogue of the Hodge filtration in de Rham cohomology.

In fact, the structure morphism $(R/A)^\nyg\to \spf(A)\times \agm$ will also factor through the stack $(\oA/A)^\nyg$ by functoriality of the algebra morphism $\oA\to R$. For intuition, in characteristic $p$, the stack $(\oA/A)^\nyg$ is just $k^\nyg$. 

This factorization then captures not just the Nygaard filtration, but also the \emph{specializations} of the Nygaard filtration; for example the Hodge filtration on the derived de Rham cohomology of $R$ relative to $\oA$ and the conjugate filtration on the Hodge-Tate complex $\HTp_{R/A}.$

In this note we specialize to the case when the prism $(A,I)=(W(k),(p))$ and study the relative derived stack $R^\nyg\to k^\nyg$ for any animated $k$-algebra $R$ with the aim of proving an analog of Theorem \ref{eq: rel prism is de rham affine} in this setting. We refer the reader to the forthcoming Remark \ref{rem: remains true in general} on the \emph{mutatis mutandis} generalization of this result to mixed characteristic pursued in \cite{Sah25Syn} and to \S~\ref{forthcoming} for an explanation of the context .

%\begin{example}
  %  In his recent course notes \cite[Section 4]{bhatt-pHT-notes}, Bhatt proves that the perfectization of an affine $p$-adic formal scheme is de Rham affine over $\spf W(k).$
%\end{example}

%\begin{remark}
    %We hasten to clarify that we take absolutely no credit for the \emph{discovery} of this property for de Rham type invariants and we do not know whom to attribute it to. Our task in this paper is to name it and prove the forthcoming Theorem \ref{thm: intro Main theorem}. 
%\end{remark}

\subsection{Main result}\label{sec: results intro}

Our main result is to show that the functor which sends an animated $k$-algebra $R$ to $R^\nyg$ is de Rham affine relative to $k^\nyg.$

\begin{notation} We recall some notation in order to formulate our theorem. We will elaborate more in  \S~\ref{sec: rees stack}.\label{not: notation intro}
\begin{enumerate}
    \item  Following \cite[\S~5.1]{APC}, there is a relative Nygaard filtration on $$F^*\Prism_{X/W(k)}:=\Prism_{X/W(k)}\widehat{\otimes}_{W(k),\varphi} W(k)$$ where $\varphi\colon W(k)\to W(k)$ is the Frobenius on Witt vectors. We denote the filtered object as $\fil_\nyg F^*\Prism_{X/W(k)}.$ 
    \item $\fil_\nyg F^*\Prism_{X/W(k)}$ is a $p$-complete  derived $\fil_{(p)}W(k)$-algebra , where  $\fil_{(p)}W(k)$ is the $p$-adic filtration on $W(k).$

    \item Recall by \cite[\S~3.3]{Bhatt22} that $k^\nyg\simeq \spf W(k)\langle u,t\rangle/(ut-p)/\GG_m$ with $t$ in degree $1$ and $u$ in degree $-1.$
    \item   By taking the Rees construction of the filtered derived algebra of item $(2)$ we get a graded derived algebra $\rees(\fil_\nyg F^*\Prism_{X/W(k)})\in \dalg(k^\nyg)$ where the right hand side is defined by right Kan extensions. 
    \item Given any filtered derived algebra $A\in \dalg(k^\nyg)$ we can define, following \cite[Construction 1.2]{MM24} $$\uspf_{k^\nyg}A:=\Maps_{\dalg(k^\nyg)}(A,-)$$ as a functor on $\dalg^\cn(k^\nyg)$, the category of derived algebras on $k^\nyg$ whose underlying modules are connective for the standard $t$-structure on $k^\nyg.$  
    \item Since $k^\nyg$ has an affine diagonal, it follows that $\uspf_{k^\nyg}A$ defines a functor on derived schemes with a map to $k^\nyg.$ 
    \item When we apply the operation of item $(5)$ to the algebra of item $(4)$ we denote the result by $\cR_\nyg(F^*\Prism_{X/W(k)}).$
\end{enumerate}

\end{notation}

We may now state our main (and only) theorem.

\begin{theorem}\label{thm: intro Main theorem}
    Let $X$ be an affine derived scheme over $k$. Then there is an equivalence of derived stacks over $k^\nyg$ $$X^\nyg\simeq \cR_{\nyg}(F^*\Prism_{X/W(k)}),$$ which is natural in $X.$

    In particular there is an isomorphism of $\stacks_{k^\nyg}$-valued functors on derived affine $k$-schemes $$(-)^\nyg\simeq \cR_\nyg(F^*\Prism_{-/W(k)}) $$
\end{theorem}

As a corollary we obtain similar isomorphisms for those stacks whose information is contained within $X^\nyg$ viewed as a stack over $k^\nyg$. We mention two which highlight the utility of Theorem \ref{thm: intro Main theorem} and may be of independent interest. We refer the reader to Corollary \ref{cor: substacks iso} for the full list.

\begin{corollary}\label{cor: hodedr isaffine}
    Let $X$ be an affine derived scheme over $k$. Then let $(X/k)^{\dR,+}$ be the Hodge filtered derived de Rham stack of $X$ and  let $\uspec_{\agm}(\rees(\fil_{\hodge} \dR_{X/k}))$ be the spectrum of the graded algebra obtained by applying the Rees construction to the Hodge filtered derived de Rham cohomology $\fil_{\hodge} \dR_{X/k}$ of $X$. 
    Then there is a canonical isomorphism of stacks over $\agm$

    $$(X/k)^{\dR,+}\simeq \uspec_{\agm}(\rees(\fil_{\hodge} \dR_{X/k})) $$
\end{corollary}

\begin{remark}
    Corollary \ref{cor: hodedr isaffine} is a filtered variant of Mondal's result in \cite{reconstruction_mondal}. We actually do not know how to prove \ref{cor: hodedr isaffine} directly via identification of the ring stack as in \emph{loc.cit.} It would be interesting to give a direct proof of this fact in the spirit of Tannaka duality. See the forthcoming Corollary \ref{cor: colimit preserve cohomology arises from transmutation} and Question \ref{remark: how to identify ring stack} for context.
\end{remark}

\begin{corollary}\label{cor: filconj is affine}
Let $X$ be an affine derived scheme over $k$ and let $(X/W(k))^{\HT,\conj}$ be the conjugate filtered Hodge-Tate stack of $X$ relative to $W(k)$. Let $\uspec_{\agm}(\rees(\fil_\conj \HTp_{X/W(k)}))$ be the spectrum of the Rees construction of the conjugate filtered Hodge-Tate cohomology $\fil_\conj \HTp_{X/W(k)}$ of $X$ relative to $W(k).$ Then there is an equivalence of stacks over $\agm$
$$(X/W(k))^{\HT,\conj}\simeq  \uspec_{\agm}(\rees(\fil_\conj \HTp_{X/W(k)})).$$
    
\end{corollary}

\begin{remark}
    After Frobenius pullback Corollary \ref{cor: filconj is affine} implies a similar result for the conjugate filtration on the de Rham cohomology of $X$ relative to $k.$ We leave this to the interested reader to explicate.
\end{remark}

\begin{remark}
    With Corollary \ref{cor: filconj is affine}, we again do not know a direct proof. It would be interesting to find one using Tannaka duality.
\end{remark}

\begin{remark}\label{rem: remains true in general}
    The theorem remains true \emph{mutatis mutandis} when one works with relative Nygaard filtered prismatization which is the subject of \cite{Sah25Syn}, as the latter hasn't been studied in detail before (but see \cite{LahotiManam2025CohomologyRingStacks} and \cite{pentland2025syntomification} for some definitions). Note that the proof can be easily made to work relative to perfect prisms, but we defer that to \cite{Sah25Syn}.
\end{remark}

\begin{remark}[Overview of the proof]\label{remark: overview proof}
The proof proceeds by building a natural map and then by reduction to the case of $k[x]$ for both the functors in Theorem \ref{thm: intro Main theorem}. Set $A=k[x]$ and $S=k[x^{1/p^\infty}]$. Using the criterion abstracted in Lemma \ref{lem: checking on a single cover}, it will suffice to show that the natural map $\cR_\nyg(F^*\Prism_{S/W(k)})\to \cR_\nyg(F^*\Prism_{A/W(k)})$ is an fpqc epimorphism of sheaves of spaces. This can be obtained from a more concrete flatness statement about the morphism of derived algebras $\fil_\nyg F^*\Prism_{A/W(k)}\to \fil_\nyg F^*\Prism_{S/W(k)}$, which is obtained via the Rees construction and a constructible stratification flatness argument as explained in Appendix \ref{appendix on testing flatness}. The main difficulty then is in setting up a theory of derived algebras on non-trivially formal derived stacks which was accomplished in \cite{dalggauges}. The remaining effort in this paper is explaining the construction of the stack $\rnyg{R}$ for an animated $k$-algebra $R.$
\end{remark}

\begin{remark}
    In \cite[Lemma 3.3.14.]{Holeman-Derived-delta-Rings}, Holeman gives another proof of \cite[Theorem 7.17]{APCII} establishing the de Rham affineness of prismatic cohomology, by using the universal property of derived prismatic cohomology established in his work. We expect that something similar will give a more conceptual proof of Theorem \ref{thm: intro Main theorem}, but we make no attempt to do so.
\end{remark}

\begin{remark}
    In the notation of Remark \ref{remark: overview proof}, it is not too hard to prove that the natural map $\fil_\nyg F^*\Prism_{A/W(k)}\to \fil_\nyg F^*\Prism_{S/W(k)}$ is descendable in the sense of Mathew \cite{galoisgroup}\footnote{We sketch a proof. It suffices to show that the induced map of graded algebras $\rees(\fil_\nyg F^*\Prism_{A/W(k)})\to \rees(\fil_\nyg F^*\Prism_{S/W(k)})$ is descendable of index $\leq 2$. Set $P=\rees(\fil_\nyg F^*\Prism_{A/W(k)})$ and $Q=\rees(\fil_\nyg F^*\Prism_{S/W(k)}).$ Let $F=\fib(P\to Q)$. Then it suffices to show that $F^{\otimes_P 2}\to P$ is nullhomotopic. Since the target is $t$-complete, and $t$-completion is exact, we may pass to $t$-completions. Note that this doesn't change $Q\qq t$ by \cite[Lemma 2.1]{GwilliamPavlov2018EnhancingFilteredDerived}. Therefore we may pass mod $t$ and show that the induced map $F^{\otimes_P 2}\qq t \to P \qq t$ is nullhomotopic. But by the completeness of the conjugate filtration on $\dR_{k[x]/k}$, we may apply the same trick and go mod $u$ where the technique of \cite[Lemma 8.6]{BS19} in the form of Lemma \ref{lem: filtrations on de rham complex} applies.}. Pending a graded analog of \cite[Proposition 3.5]{MM24}, it should imply the following result which is a direct analog of Bhatt's result in \cite[Corollary 4.21]{MM24}. Let $R$ be any animated $k$-algebra such that the $k$-module $\Omega^1_{\pi_0(R)/k}$ is finitely generated. Then the natural map $$\Gr\Mod^{\pcomp}_{\rees(\fil_\nyg F^*\Prism_{R/W(k)})}\to \qcoh(\rnyg{R})$$ is an equivalence. Since this will involve writing down a notion of \emph{graded} coconnective flatness in the sense of \emph{loc.cit.}, we do not pursue this here.
\end{remark}

\subsection{Overview of the Note.}\label{sec: overview}
In \S~\ref{sec: de Rham affine.}, we explain some basic generalities on the notion of de Rham affineness. In \S~\ref{sec: cons strat of knyg}, we recall some loci of $k^\nyg$ which are useful to state the corollaries of our theorem. In  \S~\ref{sec: rees stack} we construct the Rees stack $\rnyg{R}$ for any animated $k$-algebra $R$, study its restrictions to loci of $k^\nyg$ studied in  \S~\ref{sec: cons strat of knyg}, and construct a natural map exhibiting the isomorphism in Theorem \ref{thm: intro Main theorem}. 
Finally in  \S~\ref{sec: de Rham affine proof} we give the proof. 

The appendices fill in some details missing from the literature. In Appendix \ref{appendix on testing flatness} we explain a criterion for testing faithful flatness of connective modules on any constructible stratification of the derived affine scheme corresponding to an animated ring. In Appendix \ref{appendix: connectivity mod J}, we explain why connectivity of modules complete with respect to an ideal can be tested modulo that ideal. In Appendix \ref{Appendix: colim derived nyg}, we explain why the relative Nygaard filtered prismatic cohomology preserves colimits.

\subsection{Forthcoming work.}\label{forthcoming}

This note is the second in a series of papers dedicated to studying the syntomification (in the sense of \cite{Bhatt22}) of Frobenius liftable schemes. The other papers are \cite{dalggauges} and \cite{Sah25Syn}.
We refer the reader to \cite[\S~5]{dalggauges} for more context, but we recall here the main context of Theorem \ref{thm: intro Main theorem}. In  \cite{Sah25Syn} we will show that if $(\frak{S},I)$ is a Breuil-Kisin prism (or more generally any bounded orientable prism) then the relative Nygaard filtered prismatization $(-/\frak{S})^\nyg$ is de Rham affine in an appropriate sense. In particular when $\frak{S}/I=W(k)$ then the functor sending a derived formal $W(k)$-scheme $X\mapsto (X/W(k))^{\dR,+}$, the Hodge filtered de Rham stack, is de Rham affine relative to $\agm$ i.e. a mixed characteristic lift of the statement of Corollary \ref{cor: hodedr isaffine}. Using an observation of \cite{BhattMorrowScholze2019THHandPadicHodge}, this along with Theorem \ref{thm: main thm} will allow us to show that if $X$ is a derived $\delta$-scheme over $W(k)$ and $X_{p=0}$ is the derived mod $p$-reduction, then there is a pullback diagram % https://q.uiver.app/#q=WzAsNCxbMSwwLCIoWC9XKGspKV57XFxkUiwrfSJdLFsxLDEsIlxcYWdtIl0sWzAsMSwia15cXG55ZyJdLFswLDAsIihYX3twPTB9KV5cXG55ZyJdLFswLDFdLFsyLDFdLFszLDJdLFszLDBdXQ==
\[\begin{tikzcd}
	{(X_{p=0})^\nyg} & {(X/W(k))^{\dR,+}} \\
	{k^\nyg} & \agm
	\arrow[from=1-1, to=1-2]
	\arrow[from=1-1, to=2-1]
	\arrow[from=1-2, to=2-2]
	\arrow[from=2-1, to=2-2]
\end{tikzcd}\]
Using this, in the $p$-completely smooth case, we will describe $\qcoh((X_{p=0})^\syn)$ in terms of Fontaine-Laffaille modules on $X$ of Faltings \cite{faltings}.

\subsection{Conventions}\label{sec: conventions}
We use the language of $\infty$-categories and derived algebraic geometry from \cite{HTT}, \cite{HA} and \cite{Lur18}. As a consequence all functors and (co)limits are taken in the derived sense.

We work with animated rings and their module categories. We denote this category as $\aring$ and for $A\in \aring$ we denote animated $A$-algebras by $\aring_A:=\aring_{A/}$. If $A\in \aring$  then it has an underlying $E_\infty$-ring $A^\circ$ and a module category $\Mod_{A^\circ}:=\Mod_{A^\circ}(\Sp)$ where $\Sp$ is the $\icat$ of spectra. We define $\Mod_A$ as $\Mod_{A^\circ}.$ This is equipped with a standard $t$-structure obtained by pullback of the $t$-structure on $\Sp$ and we denote the connective part by $(\Mod_A)_{\geq 0}$ and the coconnective one by $(\Mod_A)_{\leq 0}$. 

We will have occasion to take both derived quotients and classical quotients and we differentiate between them as follows: if $A$ is an animated ring and $f_1,\ldots, f_r\in \pi_0(A)$ then $A\qq f_1,\ldots, f_r=A\otimes_{\ZZ[x_1,\ldots, x_r]} \ZZ $ where the action on the left factor is given by $x_i\mapsto f_i$ and the right factor is given by $x_i\mapsto 0.$ When $A$ is a discrete ring, then the classical quotient in the same situation is denoted $A/(f_1,\ldots, f_r).$ A similar convention holds for modules. See \cite[\S~5]{CesnaviciusScholze2024FlatPurity} for more on this.

A stack is an accessible hypersheaf $X\colon (\aring)_{\fpqc}\to \calS $ where the target is the $\icat$ of spaces or $\infty$-groupoids or anima. We denote this category simply as $\stacks.$ Given a stack $X$, we will also be interested in the $\icat$ of stacks over $X$ which we denote by $\stacks_X:=(\stacks)_{/X}$. Note that this can equivalently be defined as flat accessible hypersheaves on $\daff_X$, the $\infty$-category of derived affine schemes with a map to $X$, where $\daff=\aring^{\op}$ by definition.

The accessibility condition implies that $\qcoh(X),$ defined by right Kan extensions from $\daff_X$, is presentable (for example by \cite[Proposition 3.2.2]{dag}). Similarly we will be interested in the $\icat$ $\dalg(X)$ of derived algebras in quasi-coherent sheaves on a stack $X$. These are defined again by right Kan extension from $\daff_X$ and are monadic over $\qcoh(X)$ and presentable (by the presentability of $\qcoh(X)$) although we never need the former fact. We will be interested in the classification results for $\dalg(k^\nyg), \dalg(\agm)$ and $\dalg(\bgm)$ as explained in \cite[\S~4]{dalggauges}.

$\qcoh(X)$ has a canonical $t$-structure where a module $M\in\qcoh(X)$ is connective iff it is connective after pullback to $\spec(R)\to X$ for all $\spec(R)\in \daff_X$. The coconnective part is determined by orthogonality. This induces a notion of connectivity on $\dalg(X)$ wherein a derived algebra is connective if it is in $\qcoh(X).$ We denote the connective part by $\dalg^\cn(X).$

We will be interested in applying the adjoint functor theorem \cite[Corollary 5.5.2.9 ]{HTT} for the `large' category $\stacks_X$. For this, we adopt the easiest resolution which is to allow hypersheaves with values in large spaces $\widehat{\calS}$, defined as in \cite[Remark 1.2.16.4]{HTT}. Enlarging our universe then makes $\stacks_{X}$ an $\infty$-topos and hence presentable (in the larger universe) by \cite[Theorem 6.1.0.6.]{HTT}. Another approach via working with accessible presheaves on coaccessible categories is described in \cite[Appendix A]{HesselholtPstragowski-Dirac-geometry-II} used in \cite[Appendix D]{LahotiManam2025CohomologyRingStacks}, which we do not use here (due to our own ignorance).\footnote{Maxime Ramzi has informed the author that the application of the adjoint functor theorem does not require presentability of the target, but only the source. In particular, one does not need to enlarge the universe to apply the theorem for accessible hypersheaves. In the interest of brevity, we still do so.}

A $p$-adic formal stack is a stack on $p$-nilpotent animated rings. Here, an animated ring $R$ is $p$-nilpotent if $p^N=0$ in $\pi_0(R)$ for some $N\gg0 .$ The condition of being a $p$-adic formal stack is equivalent to being a stack $X\in \stacks$ with the additional property that the structure morphism $X\to \spec(\ZZ)$ factors through $X\to \spf(\ZZ_p)\to \spec(\ZZ)$\footnote{In fact, after fixing set theoretic issues, there is an equivalence of $\infty$-topoi $\Shv((\aring)_\fpqc)_{/\spf(\ZZ_p)}\simeq\Shv((\nilp_p)_{\fpqc})$ with the functor given by the projection $\Shv((\aring)_\fpqc)_{/\spf(\ZZ_p)}\to \Shv((\aring)_\fpqc)$ composed with the restriction $\Shv((\aring)_\fpqc)\to \Shv((\nilp_p)_{\fpqc})$ where $\nilp_p$ is the $\icat$ of $p$-nilpotent animated rings}.

A ring stack is a stack which factors through the forgetful functor $\aring\to \calS.$ We will mostly be interested in those stacks which are $1$-truncated. We refer the reader to \cite{Drinfeld2021RingGroupoid} for a gentle introduction to this theory.

In the main body of the article we work with derived variants of prismatic cohomology and its concomitant structures like derived Hodge-Tate and Hodge cohomology. We drop the word `derived' as we never mean the classical notion. 

Lastly, we work with prismatic cohomology relative to a perfect prism $(W(k),(p))$ where $k$ is a perfect ring of characteristic $p>0$. By the Bhatt-Lurie comparison \cite[Theorem 5.6.2]{APC} this is equivalent to absolute prismatic cohomology. However, since we do care about structure which needs a map from $k$, like the conjugate filtration on Hodge-Tate cohomology, we prefer to work with relative prismatic cohomology.

\subsection{Acknowledgements.}\label{sec: acknowledgements}
We thank Bhargav Bhatt, Jack J. Garzella,  Ryo Ishizuka, Keerthi Madapusi, Deven Manam, Maxime Ramzi, Longke Tang, Gleb Terentiuk and Nathan Wenger for helpful conversations or correspondences. 
In particular, Bhargav Bhatt suggested looking at the constructible stratification of $k^\nyg$, crucially suggested the bar complex for estimates on connectivity and gave us permission to name the property in Definition \ref{def: de rham affine}\footnote{All misinterpretations or unsuitability of terminology are solely the author's fault.}.
We are also extremely grateful to our advisor Kiran Kedlaya for encouragement, support and the freedom to pursue our interests without which this note would not exist in any form. The author was partially supported by NSF grant  DMS-$2401536$ and the S.E. Warschawski Professorship under Kiran Kedlaya.

Lastly, we thank the anonymous referees for detailed comments that improved the readability and precision of this paper. 

\subsection{LLM disclosure.}\label{sec: llm}
No LLMs were used in the research, writing or preparation of this paper.

\section{Generalities on de Rham affineness}\label{sec: de Rham affine.}

In this section we study some generalities on de Rham affineness as in Definition \ref{def: de rham affine} which will serve to axiomatize the notion and streamline the proof.  Our viewpoint on this property is that  it is a useful organizational tool for our work in the forthcoming \cite{Sah25Syn}. However we are able to show some structural properties like Corollary \ref{cor: colimit preserve cohomology arises from transmutation} which may be interesting in their own right.

We recall the definition for convenience.

\begin{definition}\label{def: 2de rham.}
    Let $S$ be a stack with an affine diagonal, let $\stacks_S$ be the $\icat$ of stacks over $S$ and let $\dalg(S)$ be the $\icat$ of derived algebras in quasi-coherent sheaves on $S.$ Let $A$ be an animated ring and $\aring_A$ the $\icat$ of animated $A$-algebras.
    Assume $\mathfrak{X}\colon \aring_A^{\op}\to \stacks_S$ is a functor. We say that $\mathfrak{X}$ is de Rham affine relative to $S$ if there is a functor $\cT\colon \aring_A\to \dalg(S)$ and an isomorphism of $\stacks_S$-valued functors $$\mathfrak{X}(-)\simeq \uspec (\cT(-)).$$
\end{definition}

\begin{example}\label{eg: trivial example}
    For any $\dalg(S)$ valued functor $\cT\colon \aring_A\to \dalg(S)$ the functor $\uspec(\cT(-))$ defined as the composite 

    $$\aring_A^\op\xrightarrow{\cT} \dalg(S)^\op\xrightarrow{\uspec} \stacks_S$$ is de Rham affine by definition.
\end{example}

\begin{remark}
In Definition \ref{def: 2de rham.} one thinks of $\cT\colon \aring_A\to \dalg(S)$ as a cohomology theory one wants to study and $\mathfrak{X}$ as parametrising coefficients for the cohomology theory in the sense that for any animated $A$-algebra $R$ one considers $\Qcoh(\mathfrak{X}(R))$ as the category of coefficients for $\cT$-cohomology. 

\begin{comment}

It will happen under mild finiteness conditions on $R$ that $\cT(R)=f^{\mathfrak{X}(R)}_*\cO_{\mathfrak{X}(R)}$ as derived algebras, but not necessarily always. For example if $S=\spec(Q)$ for some discrete ring $Q$, then  $$f^{\mathfrak{X}(R)}_*\cO_{\mathfrak{X}(R)}=\Gamma(\mathfrak{X}(R), \cO_{\mathfrak{X}(R)})=\lim_{\spec (V)\to \mathfrak{X}(R)}V,$$ and thus will not generically preserve colimits in $\aring_A$. 
In contrast, it will often happen that $\mathfrak{X}(-)$ will preserve colimits. In particular it is not clear whether the canonical map $$\mathfrak{X}(R)\to \uspec(\Gamma(\mathfrak{X}(R), \cO_{\mathfrak{X}(R)}))$$ will be an isomorphism in general. See the forthcoming Example \ref{eg: example of relative prismatization} and Question \ref{question: bhatt}.

\end{comment}
\end{remark}

\begin{recollection}\label{rec: transmutation}(Transmutation)
    The theory of prismatization (and its cousins) provides us with examples of $\mathfrak{X}$ coming from \emph{transmutation} \cite[Remark 2.3.8]{Bhatt22} as follows. There is an $\aring_A$-valued stack  $\cB\colon \daff^{\op}_S\to \aring_A$ so that for any animated $A$-algebra $R$, the value of $\fX(R)$ on $Z\in \daff_S$ is given by the formula
    
    $$(\mathfrak{X}(R))(Z)=\maps_{\aring_A}(R, \cB(Z)).$$

    If $\mathfrak{X}$ is of the form above, we say that it \emph{arises from transmutation.}

    Note that $\mathfrak{X}(-)$ arising from transmutation is a colimit preserving functor $\aring_A^\op\to \stacks_S.$ We have a converse.
\end{recollection}

\begin{lemma}\label{lem: limit gives transmutation}
    Let $\mathfrak{X}\colon \aring_A^\op\to \stacks_S$ be a colimit (in $\aring_A$) preserving functor. Then $\mathfrak{X}$ arises by transmutation.
\end{lemma}
\begin{proof}
This follows from the adjoint functor theorem \cite[5.5.2.9]{HTT}\footnote{We refer the reader to our conventions in \S~\ref{sec: conventions} for a fix to the implicit set theoretic issues.}. In more detail, we see that $\mathfrak{X}$ admits a left adjoint $\cB\colon \stacks_S\to \aring_A^{\op}$ with the property that 

    $$\maps_{\aring_A}(R,\cB(\cY))=\Maps_{\stacks_S}(\cY,\mathfrak{X}(R)).$$

    After restricting along the Yoneda embedding $\daff_S\into \stacks_S$ we obtain the requisite result.
\end{proof}

We want to now specialise to the case of cohomology theories of interest, in particular those which arise from prismatic theory (or cousins). 

\begin{corollary}\label{cor: colimit preserve cohomology arises from transmutation}
    Let $\cT\colon \aring_A\to \dalg(S)$ be colimit preserving. Then the composite 
    $$\aring_A^\op \xrightarrow{\cT} \dalg(S)^\op\xrightarrow{\uspec} \stacks_S$$ arises from transmutation.
\end{corollary}

\begin{proof}
    The composite is colimit preserving and so Lemma \ref{lem: limit gives transmutation} applies. We thus have an isomorphism of $\stacks_S$-valued functors
    $$\uspec (\cT(-))=\maps_{\aring_A}(-,\cB)$$ for some $A$-algebra stack $\cB.$ 
\end{proof}

\begin{remark}\label{rem: inspired by mondal}
     Corollary \ref{cor: colimit preserve cohomology arises from transmutation} is inspired by the key insight of \cite{reconstruction_mondal}. However, in general, it is not so clear how to identify the $A$-algebra stack $\cB$ in general; for example, in the prismatic theory, the stacks in question are genuinely formal stacks and the Tannakian methods of \emph{loc.cit.} are not available\footnote{In fact, since the writing of this paper, recent work of the author and J. Mundinger \cite{MundingerSahai2026} has made Tannakian methods available in \emph{derived} (and hence even classical) formal geometry. We refer the interested reader to \S~8 of \emph{loc.cit.} for the precise statements.}. 
    However, we can still ask the following question in characteristic $p$ when the stacks are not formal.
\end{remark}

\begin{question}\label{remark: how to identify ring stack}
   When $\cT=\fil_\hodge\dR_{-/k}$ is regarded as a functor from $k$-algebras to $\dalg(\agm\times \spec(k))\simeq \fil\dalg_k$, then is it possible to recover $\G_a^{\dR,+}$ as a ring stack by using explicit calculations of  $\fil_\hodge\dR_{k[x]/k}$?
    We expect the answer to be yes, but we do not pursue it here.
\end{question}

\begin{remark}
    Corollary \ref{cor: colimit preserve cohomology arises from transmutation} gives one reason why the cohomology theories arising from prismatic cohomology and cousins should arise via transmutation. 
\end{remark}

\begin{remark}
 As far as the author understands, a similar result to Lemma \ref{lem: limit gives transmutation} is independently obtained in \cite[Appendix D]{LahotiManam2025CohomologyRingStacks}, although in a different context.
\end{remark}

\begin{remark}\label{remark: preservation of colimits implies Kunneth}
    If $\cT$ preserves all colimits then $\cT$ satisfies a K\"unneth formula in the following sense: If $R\leftarrow P\rightarrow Q$ is a span of animated $A$-algebras, then the canonical morphism in $\dalg(S)$
    $$\cT(R)\otimes_{\cT(P)} \cT(Q)\to \cT(R\otimes_P Q)$$ is an isomorphism.
    Therefore the notions of transmutation and K\"unneth seem to be closely related to each other. 
\end{remark}

\begin{variant}\label{variant: spec is spf}
In the notation of Definition \ref{def: 2de rham.}, it often happens that $S$ is a formal stack, for example like $\spf(W(k))$, the $p$-adic completion of an algebraic stack $\spec(W(k)).$ More generally one could also work over formal stacks which are not completions. In this case a stack over $S$ is also canonically a formal stack with the formal direction being the same one as that of $S$. In this setting if $\cT$ is the functor from Definition \ref{def: 2de rham.}, we denote $\uspec(\cT(-))$ by $\uspf(\cT(-))$ and the meaning is clear.
\end{variant}

\begin{example}\label{eg: example of relative prismatization}
Let $(A,I)$ be a bounded prism. Let $\mathfrak{X}(-)$ be the functor of \emph{relative prismatization} of \cite[Variant 5.1]{APCII} given by $\maps_{\oA}(-,\overline{W}(-))$ where $\overline{W}$ is a natural ring scheme associated to the Cartier-Witt divisor $I\to W(S)$ where $S$ is a $p$-nilpotent animated ring.
For any animated $p$-complete \footnote{The $p$-completion is not necessary and for any animated $\oA$-algebra $R$, the functor of relative prismatization depends only on its $p$-completion $R^\wedge_p$.}$\oA$-algebra $R$ one has the relative prismatization $(R/A)^\Prism$ a $(p,I)$-formal stack over $\spf(A)$ where the topology on $A$ is $(p,I)$-adic. In \cite[Theorem 7.20]{APCII} the authors prove that under certain finiteness conditions on  $R$ the derived $(p,I)$-complete $A$-algebras $\Gamma((R/A)^\Prism, \cO)$ and $\Prism_{R/A}$ agree.

But since $\Gamma((R/A)^\Prism, \cO)$ is naturally a limit, it is not clear if the functor $\Gamma((-/A)^\Prism, \cO)\colon \aring_{\oA}\to \dalg_A^{(p,I)\text{-comp}}$ preserves colimits. In contrast $\Prism_{R/A}$ preserves all colimits in animated $\oA$-algebras (e.g. by \cite[Remark 4.1.8]{APC}) and therefore is the correct candidate for the comparison $$(R/A)^\Prism\to \uspf (\Prism_{R/A})$$ as stacks over $\spf(A)$ which is shown to be an isomorphism in \cite[Theorem 7.17]{APCII}.

\end{example}

\begin{question}[Bhatt]\label{question: bhatt}
We use the notation and context of Example \ref{eg: example of relative prismatization}. Even though one doesn't expect the $(p,I)$-complete derived $A$-algebras $\Prism_{R/A}$ and $\Gamma((R/A)^\Prism, \cO)$ to be equivalent in general, do they corepresent equivalent functors over connective $(p,I)$-complete derived $A$-algebras?
In other words are the functors $\Maps_{\dalg_A^{(p,I)\text{-comp}}}(\Prism_{R/A},-)$ and $\Maps_{\dalg_A^{(p,I)\text{-comp}}}(\Gamma((R/A)^\Prism, \cO),-)$ equivalent, via the natural map, after restriction to $(\dalg_A^{(p,I)\text{-comp}})^\cn$?

We expect the answer to be no, and note that it likely amounts to producing a `large' animated algebra where $\Prism_{R/A}$ and $\Gamma((R/A)^\Prism,\cO)$ are not equivalent. In fact it suffices to do this computation for the de Rham stack as it is a specialization of the relative prismatization. 
\end{question}

We next want to give a criterion for finding a map $\mathfrak{X}(-)\to \uspec (\cT(-))$ in general. This will be useful in the sequel.

\begin{construction}\label{cons: can map}
Let $\cC\subset \aring_A$ be a full subcategory with the property that
\begin{enumerate}
\item $\cC$ generates $\aring_A$ under sifted colimits,
    \item and the values of $\cT$ on $R\in \aring_A$ are left Kan extended from $\cC.$ 
\end{enumerate}

Assume that for $R\in \cC$ there is a natural morphism $$\alpha^\cC_R\colon \cT(R)\to  f_*^{\mathfrak{X}(R)}\cO_{\mathfrak{X}(R)}.$$

Then by the universal property of left Kan extensions we obtain a map $$\alpha_{R}\colon \cT(R)\to f_*^{\mathfrak{X}(R)}\cO_{\mathfrak{X}(R)}$$ for all $R\in \aring_A$ and by adjunction, we obtain a map $$\eta_R\colon \mathfrak{X}(R)\to \uspec (\cT(R)).$$
\end{construction}

\begin{remark}\label{remark: testing on A[x]}
    Let $A$ be a discrete ring and assume that both $\cT$ and $\mathfrak{X}$ preserve all colimits in $\aring_A.$ Then the isomorphism $\eta_R$ of Construction \ref{cons: can map} can be tested on the polynomial ring $A[x].$ Indeed, note that preserving colimits implies both sides of the source and target can be reduced to finitely generated polynomial rings over $A$. Now preservation of colimits also implies K\"unneth for $\cT(R)$ by Remark \ref{remark: preservation of colimits implies Kunneth} and so we can reduce the right hand side to just $\cT(A[x])$. Similarly we may reduce fiber products for $\mathfrak{X}$ to the case of $\mathfrak{X}(A[x]).$
\end{remark}

\begin{warning}
Even though Remark \ref{remark: testing on A[x]} implies that the isomorphism can be tested on $A[x]$, the morphism $\eta_R$ of Construction \ref{cons: can map} has to be defined naturally in $R\in \aring_A$. This is because $\mathrm{End}_{\aring_A}(A[x])$ will be very non-trivial in general.  
\end{warning}

We next give a criterion for testing that the canonical map of Construction \ref{cons: can map} is an isomorphism in the case $A$ is a discrete ring.

\begin{lemma}\label{lem: checking on a single cover}
Let $A$ be discrete and $\cT$ and $\mathfrak{X}$ be colimit preserving along with   a natural in $R$ map%\footnote{Note that this map in general need not come from Construction \ref{cons: can map}. All we need is \emph{a} natural map.} 
$$\eta_R\colon \mathfrak{X}(R)\to \uspec(\cT(R)).$$  Assume further that there exists a discrete $A$-algebra $(A[x])_\infty$ and a map of $A$-algebras $A[x]\to (A[x])_\infty$ with the following properties
\begin{enumerate}

 \item  Both $\mathfrak{X}(-)$ and $\uspec(\cT(-))$ send $A[x]\to (A[x])_\infty$  to an fpqc epimorphism in the $\infty$-topos $\shv(\daff_S)$, 
 \item Let  $(A[x])_{\infty}^{\otimes_{A[x]} i+1}$ be the $i+1$-fold tensor product of $(A[x])_\infty$ over $A[x]$. Assume that for each $i\geq 0$ the canonical map  $\eta_{((A[x])_{\infty}^{\otimes_{A[x]} i+1}}$ is an isomorphism.
\end{enumerate}
Then $\eta_R$ is an isomorphism for every $R\in \aring_A.$ 
\end{lemma}
\begin{proof}
    By Remark \ref{remark: testing on A[x]} we may reduce to checking that $\eta_{A[x]}$ is an isomorphism. Set $P=A[x]$ and $Q=(A[x])_\infty$. Let $Q^*$ be the Cech conerve of $P\to Q$. Then since $\mathfrak{X}(-)$ preserves all colimits, it follows that the induced map of simplicial objects $\mathfrak{X}(Q^*)\to \mathfrak{X}(P)$ is the Cech nerve of the morphism $\mathfrak{X}(Q)\to \mathfrak{X}(P)$. A similar remark applies to the induced map $\uspec(\cT(Q^*))\to \uspec(\cT(P))$.
    
    Thus it follows that we have a commutative diagram of augmented simplicial objects where both the vertical arrows are Cech nerves of their augmentation, with the augmentation being an fpqc epimorphism by hypothesis.
% https://q.uiver.app/#q=WzAsNCxbMiwxLCJcXHVzcGVjIChcXGNUKFApKSJdLFswLDEsIlxcZlgoUCkiXSxbMCwwLCJcXGZYKFFeKikiXSxbMiwwLCJcXHVzcGVjIChcXGNUKFFeKikpIl0sWzIsMV0sWzMsMF0sWzEsMCwiXFxldGFfUCJdLFsyLDMsIlxcZXRhX3tRXip9Il1d
\[\begin{tikzcd}
	{\mathfrak{X}(Q^*)} && {\uspec (\cT(Q^*))} \\
	{\mathfrak{X}(P)} && {\uspec (\cT(P))}
	\arrow["{\eta_{Q^*}}", from=1-1, to=1-3]
	\arrow[from=1-1, to=2-1]
	\arrow[from=1-3, to=2-3]
	\arrow["{\eta_P}", from=2-1, to=2-3]
\end{tikzcd}\]

Now the top horizontal arrow is an isomorphism in each degree by hypothesis, whence by taking colimits we conclude.
\end{proof}

\begin{remark}
    We expect that the notion of de Rham affineness is closely related to the notion of unwinding ring stacks presented in \cite[\S~2.4]{LiMondal-Endomorphisms-deRham}, but we have not pursued a precise connection at this point.
\end{remark}

\section{Some loci of $k^\nyg$}\label{sec: cons strat of knyg}

This section is essentially expository, but will be useful in stating the corollaries to our theorem.

Recall that if $R$ is an animated $k$-algebra then $R^\nyg$ is obtained via transmutation through a certain ring stack $\G_a^\nyg.$  We briefly recall the construction of $\GG_a^\nyg$ below.

\begin{construction}

\begin{figure}[h]
    \centering
\[\begin{tikzcd}[cramped]
	0 & {\mathbf{G}_a^\sharp} & W & {F_*W} & 0 \\
	0 & {\mathbf{V}(L)^\sharp} & {M_u} & {F_*W} & 0 \\
	0 & {\mathbf{G}_a^\sharp} & W & {F_*W} & 0 \\
	&& {}
	\arrow[from=1-1, to=1-2]
	\arrow[from=1-2, to=1-3]
	\arrow["{u^\sharp}", from=1-2, to=2-2]
	\arrow["p"'{pos=0.2}, color={rgb,255:red,255;green,51;blue,61}, curve={height=18pt}, from=1-2, to=3-2]
	\arrow["F", from=1-3, to=1-4]
	\arrow[from=1-3, to=2-3]
	\arrow["p"'{pos=0.2}, color={rgb,255:red,255;green,51;blue,61}, curve={height=18pt}, from=1-3, to=3-3]
	\arrow[from=1-4, to=1-5]
	\arrow["{\mathrm{id}}", from=1-4, to=2-4]
	\arrow["p"'{pos=0.2}, color={rgb,255:red,255;green,51;blue,61}, curve={height=18pt}, from=1-4, to=3-4]
	\arrow[from=2-1, to=2-2]
	\arrow[from=2-2, to=2-3]
	\arrow["{t^\sharp}", from=2-2, to=3-2]
	\arrow[from=2-3, to=2-4]
	\arrow["{d_{u,t}}", from=2-3, to=3-3]
	\arrow[from=2-4, to=2-5]
	\arrow["p", from=2-4, to=3-4]
	\arrow[from=3-1, to=3-2]
	\arrow[from=3-2, to=3-3]
	\arrow["F", from=3-3, to=3-4]
	\arrow[from=3-4, to=3-5]
\end{tikzcd}\]

\caption{The diagram of $W$-module schemes over $k^\nyg$ \cite[Diagram 3.3.2]{Bhatt22}.}
    \label{fig: W-module schemes}
\end{figure}

In Figure \ref{fig: W-module schemes} we have a diagram of $W$-module schemes over $k^\nyg.$ We now explain the terms in the diagram. 

\begin{itemize}
    \item First note that $k^\nyg$ has a moduli description.  Indeed, for any $p$-nilpotent (discrete) $W(k)$-algebra $R$, the $1$-groupoid $k^\nyg(R)$ is the moduli of triples $(M, u,t)$ where $M\in \Pic(R)$ and $u\colon R\to M$ and $t\colon M\to R$ are sections and cosections so that the composite is $t\circ u=p$. 

This induces a universal line bundle $L\in \Pic(k^\nyg)$ and a universal section and cosection $$\cO\xrightarrow{u} L\xrightarrow{t} \cO$$ with composite $t\circ u=p.$

\item 
After taking divided power completions along the zero section we get a map 
$$\G_a^\sharp\xrightarrow{u^\sharp} \mathbf{V}(L)^\sharp\xrightarrow{t^\sharp} \G_a^\sharp$$ with composite  $t^\sharp\circ u^\sharp=p$ still by functoriality. This defines the leftmost vertical arrow in Figure \ref{fig: W-module schemes}. 
\item The middle group scheme $M_u$ is defined as the pushout of the span $$\mathbf{V}(L)^\sharp\xleftarrow{u^\sharp}\G_a^\sharp \xrightarrow{\mathrm{can}} W$$ where the map on the right is given via the identification  $\G_a^\sharp=W[F]$ where $W[F]$ is the Frobenius kernel of $W.$

\item The map $d_{u,t}\colon M_u\to W$ is given by the universal property of the pushout. 

\item The rightmost diagram is self-explanatory. Note that the map $M_u\to F_*W $ exists again by the universal property of the pushout: the map $\mathbf{V}(L)^\sharp\to F_*W$ is $0.$

\item One checks that the map $d_{u,t}\colon M_u\to W$ is a quasi-ideal in the sense of  \cite[\S~3.3]{Drinfeld2021RingGroupoid} and defines $$\GG_a^\nyg:=\Cone(M_u\xrightarrow{d_{u,t}} W).$$

\item If $R$ is an animated $k$-algebra, then the $p$-adic formal stack $R^\nyg$ as a stack over $k^\nyg$ can be described by a simple formula 

$$R^\nyg=\Maps_{k}(R,\GG_a^\nyg(-))$$ where the mapping space is taken in animated rings over $k.$

\end{itemize}

\end{construction}

\begin{remark}
 In Figure \ref{fig: W-module schemes}, one thinks of the leftmost vertical arrows $u^\sharp$ and $t^\sharp$ as Cartier divisors on $k^\nyg$ which can be turned on or off by either setting them to $0$, inverting them or even doing nothing. The cone presentation $M_u\xrightarrow{d_{u,t}}W$ then interacts with these parameters and one obtains various specializations of the ring stack $\GG_a^\nyg$ on the stack $k^\nyg$ as one changes these parameters. Via the general machinery of transmutation one obtains specializations of $R^\nyg$ by restriction. We now explain the loci we need below.   
\end{remark}

\begin{construction}(The locus $k^\crys$).\label{const: cryst} The locus $k^\crys$ is obtained as $k^\nyg_{t\neq 0}=\spf(W(k))$. Via transmutation the structure on prismatic cohomology determined on this locus can be understood by the restriction of the ring stack $(\GG_a^\nyg)_{t\neq 0}.$ In this case we see that $(\GG_a^\nyg)_{t\neq 0}=F_*W/p=\GG_a^\mathrm{crys}$ since
$$\fib(\GG_a^\nyg\to F_*W/p)=\mathrm{cone}(\mathbf{V}(L)^\sharp\xrightarrow{t^\sharp} \GG_a^\sharp)$$ and the latter is zero when $t\neq 0.$  Thus if $R$ is an animated $k$-algebra, then $(R)^\nyg_{t\neq 0}=(R/W)^\mathrm{crys}$ is the stack parametrising crystalline cohomology of $R$ i.e. the \emph{crystallization} of $R$ (cf. \cite[Remark 2.5.12]{Bhatt22} for the definition of this functor).
\end{construction}

\begin{construction}\label{const prism}(The locus $k^{\Prism}$)
    When $u\neq 0$, the map $\GG_a^\sharp\xrightarrow{u^\sharp} \mathbf{V}(L)^\sharp$ is an isomorphism. A simple computation with the pushout description of $M_u$ shows that $$\fib(W/p\to \GG_a^\nyg)=\Cone(\GG_a^\sharp\xrightarrow{u^\sharp} \mathbf{V}(L)^\sharp)$$ and the target is $0$ since $u^\sharp$ is an isomorphism. Hence the map $W/p\to \GG_a^\nyg$ is an isomorphism. Thus $\GG_a^\nyg|_{u\neq 0}\simeq W/p$ i.e. the ring stack parametrizing the prismatization in positive characteristic. 
\end{construction}

\begin{construction}\label{cons: kconj}
    (The locus $k^{\HT,\conj}$) This is obtained as $\knyg_{t=0}=\agm$, in which case the diagram of Figure \ref{fig: W-module schemes} specializes to the diagram of \cite[Construction 2.7.11]{Bhatt22} and the ring stack $\G_a^\nyg$ restricts to $F_*^{-1}\G_a^{\dR,\conj, W}$ of \emph{loc.cit.} For any animated $k$-algebra $R$, the stack one gets is $(R/W)^{\HT,\conj}:=\varphi_*(R/k)^{\dR,\mathrm{conj}}$, parametrising the conjugate filtration on Hodge-Tate cohomology, which comes equipped with a map to $\agm.$
\end{construction}

\begin{construction}\label{cons: khodge}(The locus $k^{\dR,+}$). This is obtained by restriction to the locus $k^\nyg_{u=0}=\agm$ whence the middle vertical quasi-ideal restricts to $$(\G_a^\nyg)_{u=0}=\Cone(\mathbf{V}(\cO(-1))^\sharp\oplus F_*W\xrightarrow{(t^\sharp,V)} W)\simeq \G_a^{\dR,+}.$$
Any animated $k$-algebra $R$ gives, by transmutation, the stack $(R/k)^{\dR,+}$, parametrising the Hodge filtered derived de Rham cohomology of $R.$  
\end{construction}

\begin{construction}\label{const HT}(The locus $k^{\HT}$).\
    The locus $k^{\HT}$ is obtained as the restriction $k^\nyg_{t=0,u\neq 0}=\spec(k)$. We can analyze the restriction of the ring stack on this locus in $2$-steps. When $u\neq 0$ then $(\GG_a^\nyg)_{u\neq 0}=W/p$ from Construction \ref{const prism}.

    Now the restriction to the locus $t=0$ amounts to reducing mod $p$ and we recover the stack $W/p=F^{-1}_*\GG_a^{\dR}.$ This ring stack has an additional feature owing to the fact that $\pi_0\GG_a^\dR=F_*\GG_a$ using derived deformation theory as we now explain.

    There is a canonical morphism $\GG_a^{\dR}\to \pi_0\GG_a^{\dR}=F_*\GG_a$ of sheaves of animated rings, which (after inverting $F_*$) via transmutation gives for any animated $k$-algebra $R$, a morphism $R^{\HT}:=(R^{(-1)}/k)^{\dR}\to \spec (R)$ which, when $R$ is smooth over $k$, is a gerbe banded in $\mathbf{V}(T_{R/k})^\sharp$ and splits whenever $R$ admits a $p$-completely flat $\delta$-lift $\widetilde{R}/W$.  We call this gerbe the \emph{Hodge-Tate} gerbe following \cite[Remark 1.1]{APCII}.
\end{construction}

\begin{construction}(The locus $k^{\mathrm{Hodge}}$)\label{const: hodge}.  This locus is obtained as the restriction $k^\nyg_{u=t=0}\simeq \bgm$. This locus lives over $B\GG_m$ and one checks that $$\GG_a^\mathrm{Hodge}=\mathrm{cone}(\mathbf{V}(L)^{\sharp} \oplus F_*W \xrightarrow{(0,V)}W)\simeq \Cone (\mathbf{V}(L)^\sharp\xrightarrow{0}\GG_a)=\GG_a\oplus B\mathbf{V}(L)^\sharp$$ which is the graded of the Hodge filtration on $\GG_a^{\dR}$ and hence for any animated $k$-algebra $R$ we learn that $(R/k)^\hodge$ is the graded stack corresponding to the Hodge filtration on its de Rham cohomology.    
\end{construction}

\section{The Rees stack
  \texorpdfstring{$\cR_\nyg(F^*\Prism_{R/W(k)})$}{R_nyg(F^*Delta_{R/k})}}
\label{sec: rees stack}

In this section our goal is to explain the basic functorialities of the stack $\cR_\nyg(F^*\Prism_{R/W(k)})$ of item $(7)$ in Notation \ref{not: notation intro} for an animated $k$-algebra $R$ and to explain its restriction to the various loci of \S~\ref{sec: cons strat of knyg}. 

We recall the construction of the stack in Construction \ref{cons: rees stack of rnyg}. But before that we recall the Rees algebra construction for derived algebras, discussed in \cite[\S~4]{dalggauges}, in our context.

\begin{construction}\label{cons: rees cons for filnyg}
Let $\fil\dalg_{\fil_{(p)}W(k)}^{\pcomp}$ be the $\icat$ of $p$-complete filtered derived algebras over the filtered ring $\fil_{(p)}W(k)$ equipped with its $p$-adic filtration. 
    
If $R$ is an animated ring over $k$, then $\fil_\nyg F^*\Prism_{R/W(k)}\in \fil\dalg_{\fil_{(p)}W(k)}^{\pcomp}.$ Indeed, the fact that $\fil_\nyg F^*\Prism_{R/W(k)}$ receives an $E_\infty$-map from $\fil_{(p)}W(k)$ is formal. The derived algebra structure can be deduced from the proof of \cite[Proposition 5.1.1]{APC} and will be explained in slightly more detail in \cite{Sah25Syn}\footnote{We sketch the argument. In the proof of \cite[Proposition 5.1.1]{APC}, the functor $\fil_\nyg(F^*\Prism_{-/W(k)})$ is left and right Kan extended from the category $\cC$ of `large' quasisyntomic (hence discrete) $k$-algebras. Here a $k$-algebra $R$ is `large' if it is generated by the image of $R^\flat:=\lim_{x\to x^p}R$ under the canonical first coordinate map $R^\flat\to R$. The values of $\fil_\nyg(F^*\Prism_{-/W(k)})$ on $\cC$ are discrete and since  $\fil\dalg_{\fil_{(p)}W(k)}^{\pcomp}$ is presentable, the left and right Kan extensions are computed in  $\fil\dalg_{\fil_{(p)}W(k)}^{\pcomp}$. }.

Now as explained in \cite[\S~4]{dalggauges}, there is a chain of equivalences of $\icats$
$$\dalg(k^\nyg)\simeq \fil\dalg_{\fil_{(p)}W(k)}^{\pcomp}\simeq \Gr\dalg_{\wkut}^\pcomp,$$ where the rightmost category denotes the $\icat$ of $p$-complete graded derived algebras over the graded $p$-complete ring $\wkut.$ 
In what follows we identify the leftmost and rightmost categories implicitly and denote the image of $\fil_\nyg F^*\Prism_{R/W(k)}$ in either as $\rees(\fil_\nyg F^*\Prism_{R/W(k)}).$\footnote{In \cite{dalggauges} we use the notation $\pi_!\rees(\fil_\nyg F^*\Prism_{R/W(k)})$ to denote the comodule structure over the coaction given by the bicommutative bialgebra $W(k)[\ZZ]$, in other words, to remember the grading. We drop that notation for now since no confusion can arise.}

Explicitly, we have $$\rees(\fil_\nyg F^*\Prism_{R/W(k)})=\widehat{\bigoplus_{i\in \ZZ}}\fil_\nyg^iF^*\Prism_{R/W(k)}t^{-i}$$ as a graded $p$-complete $\wkut$-algebra.
\end{construction}

We next prove some functorialities of the construction \ref{cons: rees cons for filnyg} across various specializations. Note that in the following lemma we have canonically trivialised the Breuil-Kisin twist (as we work in characteristic $p$) and all base changes are graded.

\begin{lemma}\label{lem: functorialities for rees cons}
     For any animated $k$-algebra $R$ the graded $p$-complete derived $\wkut$-algebra $\rees(\fil_\nyg F^*\Prism_{R/W(k)})$ has the following features across various specializations of $\wkut$ 
     \begin{enumerate}

    \item After $p$-complete graded base change along $\wkut\to (\wkut)[1/t]$ we obtain $F^*\Prism_{R/W(k)}.$
      \item After $p$-complete graded base change along $\wkut\to (\wkut)[1/u]$ we obtain $\Prism_{R/W(k)}.$

      \item After graded base change along $W\langle u,t\rangle /(ut-p)\to k[t]$ one obtains $$\rees(\fil_{\mathrm{Hodge}} \dR_{R/k})=\bigoplus_{i\in \ZZ} \fil^i_{\hodge}\dR_{R/k}t^{-i}$$ where $\fil_{\hodge}\dR_{R/k}$ is the derived de Rham complex of $R$ relative to $k$ equipped with its Hodge filtration.

    \item After graded base change along $W\langle u,t\rangle /(ut-p)\to k[u]$ one obtains $$\rees(\fil_{\mathrm{conj}} \overline{\Prism}_{R/W(k)})=\bigoplus_{i\in \ZZ} \fil^i_{\conj}\overline{\Prism}_{R/W(k)}u^{i}$$ where $\overline{\Prism}_{R/W(k)}$ is the Hodge-Tate complex of $R$ relative to $W(k)$ equipped with its conjugate filtration.

    \item After the composite graded base change  $W\langle u,t\rangle /(ut-p)\to k[u]\to k[u,u^{-1}]$ we obtain $\HTp_{R/W(k)}.$
        \item After graded base change  along $W\langle u,t\rangle /(ut-p)\to k$ one obtains the graded algebra $\bigwedge^* L_{R/k}[-*]$.

     \end{enumerate}
\end{lemma}
\begin{proof}
These are all consequences of the standard behaviour of the Nygaard filtered prismatization.

\begin{enumerate}
    \item This follows from inverting $t$ in $\rees(\fil_\nyg F^*\Prism_{R/W(k)})$ and using the equivalence between $p$-complete graded $\wkut[1/t]\simeq W(k)\langle t,t^{-1}\rangle$-modules and $p$-complete $W(k)$-modules by taking the degree $0$ piece. 
    \item This follows from the Bhatt-Lurie localization for relative Frobenius\footnote{But perhaps in characteristic $p$, this goes back to \cite{FontaineJannsen-Frobenius-Gauges-I}.} \cite[Corollary 5.2.15]{APC} which we reprove in our language in the forthcoming Lemma \ref{lem: Bhatt-Lurie localisation}.

    \item Note that $u=pt^{-1}$. Setting $u=0$ in  $$\rees(\fil_\nyg F^*\Prism_{R/W(k)})=\widehat{\bigoplus_{i\in \ZZ}}\fil_\nyg^iF^*\Prism_{R/W(k)}t^{-i}$$ replaces the coefficients of $t^{-i}$ with $$\cofib(\fil^{i-1}_\nyg F^*\Prism_{R/W(k)}\xrightarrow{p}\fil^{i}_\nyg F^*\Prism_{R/W(k)})\simeq \fil^{i}_{\hodge}\dR_{R/k},$$ where the equivalence comes from the filtered de Rham comparison for relative Nygaard filtered prismatic cohomology from \cite[Corollary 5.2.8.]{APC}, which then gives us the result. 

    \item This follows from noting that $(\gr_n)_\nyg F^*\Prism_{R/W(k)}\simeq \fil^n_{\conj}\HTp_{R/W(k)}$, as explained in \cite[Remark 5.1.2]{APC}. Note that the conjugate filtration is an increasing filtration whence the explicit formula $$\rees(\fil_{\mathrm{conj}} \overline{\Prism}_{R/W(k)})=\bigoplus_{i\in \ZZ} \fil^i_{\conj}\overline{\Prism}_{R/W(k)}u^{i}$$ has the same sign for the exponent of the variable $u$ and the piece of the filtration.
    \item This amounts to taking the underlying object of the conjugate filtration on Hodge-Tate cohomology.
    \item This is a consequence of the fact that the graded of the conjugate filtration is Hodge cohomology.
\end{enumerate}

\end{proof}

The following lemma was used in the proof of item $(2)$ in the proof of Lemma \ref{lem: functorialities for rees cons}

\begin{lemma}[Bhatt-Lurie localization for relative Frobenius]\label{lem: Bhatt-Lurie localisation}
    Let $\varphi\colon \fil^\bullet_\nyg F^*\Prism_{R/W(k)}\to p^\bullet \Prism_{R/W(k)}$ be the relative Frobenius on $\Prism_{R/W(k)}.$ Then the induced map on $\ZZ$-graded $p$-complete derived algebras $$\rees(\fil^\bullet_\nyg F^*\Prism_{R/W(k)})=\widehat{\bigoplus_{i\in \ZZ}}\fil^i_\nyg F^*\Prism_{R/W(k)}t^{-i}\to   \widehat{\bigoplus_{i\in \ZZ}} (p)^i \Prism_{R/W(k)}t^{-i}$$ is a $p$-complete localization at one element. Here the target has graded pieces $(p)^i\otimes \Prism_{R/W(k)}t^{-i}$ for each $i\in \ZZ$.
\end{lemma}
\begin{proof}
First, we may test this at the level of the underlying $E_\infty$-rings.
   Now, since localization commutes with colimits, by left Kan extensions we may reduce to the case where $R=k[x_1,\ldots, x_d]$. In this case, the element $u=pt^{-1}\in \wkut$ is invertible in the target whence we are reduced to checking that the induced morphism 

   $$(\widehat{\bigoplus}_{i\in \ZZ}\fil^i_\nyg F^*\Prism_{R/W(k)}t^{-i})_{u}\to   \widehat{\bigoplus}_{i\in \ZZ} p^i \Prism_{R/W(k)}t^{-i}$$ is an isomorphism. 

   The source and target are graded algebras over the graded ring $\wkut[1/u]\simeq W(k)\langle u, \frac{p}{u}\rangle\simeq W(k)\langle u, u^{-1}\rangle$ and so the morphism is determined by its degree $0$-part.

   On the source we have in degree $0$
   $$(\widehat{\bigoplus}_{i\in \ZZ}\fil^i_\nyg F^*\Prism_{R/W(k)}t^{-i})_{u}^0=\colim(F^*\Prism_{R/W(k)}\xrightarrow{pt^{-1}} \fil^1_\nyg F^*\Prism_{R/W(k)}t^{-1}\xrightarrow{pt^{-1}} \fil^2_\nyg F^*\Prism_{R/W(k)}t^{-2}\xrightarrow{pt^{-1}} \ldots) $$ and similarly on the target we have

    $$(\widehat{\bigoplus}_{i\in \ZZ} p^i\Prism_{R/W(k)}t^{-i})^0=\colim(\Prism_{R/W(k)}\xrightarrow{pt^{-1}} p\Prism_{R/W(k)}t^{-1}\xrightarrow{pt^{-1}} p^2 \Prism_{R/W(k)}t^{-2}\xrightarrow{pt^{-1}} \ldots),$$ where the transition maps in the second diagram are isomorphisms. The relative Frobenius induces a morphism between the two diagrams above.

    Now the first diagram stabilizes at $i=d$ by \cite[Corollary 5.2.15.]{APC} whence it suffices to prove that the relative Frobenius $$\fil^d(\varphi)\colon \fil^d_\nyg F^*\Prism_{R/W(k)}\to p^d \Prism_{R/W(k)}$$ is an isomorphism. This is proved in the proof of \cite[Corollary 5.2.16]{APC}, but we recall their proof for the reader's convenience. 
    Since both the source and target are $p$-complete it suffices to check the isomorphism mod $p$. 
    By  \cite[Corollary 5.2.15.]{APC} we may identify $(p)\otimes \fil^{d}_\nyg F^*\Prism_{R/W(k)}\simeq \fil^{d+1}_\nyg F^*\Prism_{R/W(k)} $, whence  $\fil^d_\nyg F^*\Prism_{R/W(k)}\qq p$ agrees with $\gr^d_{\nyg} F^*\Prism_{R/W(k)}$ and so the mod $p$-reduction of the relative Frobenius is given by $$\gr^d(\varphi)\colon \gr^d_\nyg F^*\Prism_{R/W(k)}\to \HTp_{R/W(k)}.$$ By \cite[Remark 5.1.2]{APC}, this map identifies with the tautological map $$\fil^d_{\conj} \HTp_{R/W(k)}\to \HTp_{R/W(k)}.$$
    Since the conjugate filtration is exhaustive, it suffices to prove that $\gr^{d+i}_{\conj} \HTp_{R/W(k)}$ vanishes for $i\geq 1.$ This is a consequence of the Hodge-Tate comparison
    $$\gr^{d+i}_{\conj} \HTp_{R/W(k)}\simeq L\widehat{\Omega}^{d+i}_{R/k}[-d-i]$$ and the fact that the right hand side is zero when $i\geq 1$ as $R=k[x_1,\ldots, x_d].$

\end{proof}

\begin{remark}\label{remark: specialization preserve derived algebras}
    Note that all the specializations in Lemma \ref{lem: functorialities for rees cons} preserve derived algebra structures as they are obtained by (completed) base changes.
\end{remark}

\begin{construction}[The Rees stack $\rnyg{R}$]\label{cons: rees stack of rnyg}
For any animated $k$-algebra $R$, we may consider $\rees(\fil_\nyg F^*\Prism_{R/W(k)})\in \dalg(k^\nyg)$ via the equivalences explained in Construction \ref{cons: rees cons for filnyg}. Then we define the stack $$\uspf_{k^\nyg}(\rees(\fil_\nyg F^*\Prism_{R/W(k)})):=\maps_{\dalgk}(\rees(\fil_\nyg F^*\Prism_{R/W(k)}),-)$$ and for ease of notation 
$$\cR_\nyg(F^*\Prism_{R/W(k)}):=\uspf_{k^\nyg}(\rees(\fil_\nyg F^*\Prism_{R/W(k)})).$$

Note that $k^\nyg$ has an affine diagonal, whence any affine scheme $f\colon \spec (A)\to k^\nyg$ can be incarnated as a connective derived algebra $f_*\cO_{\spec(A)}\in \dalg^\cn(k^\nyg)$ and so it makes sense to think of $\rnyg{R}$ as a stack over $k^\nyg$. Note that $\cR_\nyg(F^*\Prism_{R/W(k)})$ is a $p$-adic formal stack as the structure morphism  $\rnyg{R}\to \spec(\ZZ)$ factors over $\spf(\ZZ_p)\to \spec(\ZZ).$ 
    
\end{construction}

For any animated $k$-algebra $R$ we may specialize the Rees stack $\rnyg{R}$ on the various substacks of $k^\nyg.$ It is therefore useful to name the restrictions as we do so in the next construction.

\begin{construction}[Specializations of the Rees stack $\rnyg{R}$]\label{cons: specialisation of rees stacks}
Fix an animated $k$-algebra $R$. For the various identifications of $\dalg$ on prestacks we refer the reader to \cite[\S~4]{dalggauges}.
Using the notation in the statement of Lemma \ref{lem: functorialities for rees cons} we have

\begin{enumerate}
    \item Over $\spf(W(k))$ we have $F^*\Prism_{R/W(k)}\in \dalg(\spf(W(k)))=\dalg^{\pcomp}_{W(k)}$ and we set $$\uspf(F^*\Prism_{R/W(k)})=\Maps_{\dalg(\spf(W(k)))}(F^*\Prism_{R/W(k)},-).$$

    \item In the same situation as the previous item we set $$\uspf(\Prism_{R/W(k)})=\Maps_{\dalg(\spf(W(k)))}(\Prism_{R/W(k)},-).$$

    \item We may consider $\rees(\fil_{\hodge} \dR_{R/k})\in \dalg(\agm)=\fil\dalg_{k}$ and set $$\rhge{R}:=\uspec_{\agm}(\rees(\fil_{\hodge} \dR_{R/k}))$$ where the right hand side again is $\Maps_{\dalg(\agm)}(\rees(\fil_{\hodge} \dR_{R/k}),-).$

    \item In the same situation as before we define $$\rconj{R}:=\uspec_{\agm}(\rees(\fil_{\conj} \HTp_{R/W(k)})).$$

    \item We may consider $\HTp_{R/W(k)}\in \dalg(\spec(k))=\dalg_k$ and set $$\uspec(\HTp_{R/W(k)})=\maps_{\dalg(\spec(k))}(\HTp_{R/W(k)},-).$$

    \item We may consider $\bigwedge^* L_{R/k}[-*]\in \dalg(\bgm)=\Gr\dalg_k.$
    Then we set 
    $$\uspec_{\bgm}(\bigwedge^* L_{R/k}[-*]):=\maps_{\dalg(\bgm)}(\bigwedge^* L_{R/k}[-*],-).$$

\end{enumerate}
    
\end{construction}

We next have the obligatory lemma recording the behaviour of $\rnyg{R}$ by pullback to the substacks of \S~\ref{sec: cons strat of knyg}.

\begin{lemma}
Pullbacks of $\rnyg{R}$ along the loci of $\knyg$ are given as follows
\begin{enumerate}
    \item $\rnyg{R}$ pulls back to $\uspf(F^*\Prism_{R/W(k)})$ along the locus $\spf(W(k))=k^\crys\into k^\nyg$.
  \item $\rnyg{R}$ pulls back to $\uspf(\Prism_{R/W(k)})$ along the locus $\spf(W(k))=k^\Prism\into k^\nyg$.
  \item $\rnyg{R}$ pulls back to $\rhge{R}$ along the locus $\agm=k^{\dR,+}\into k^\nyg$.
  \item $\rnyg{R}$ pulls back to $\rconj{R}$ along the locus $\agm=k^{\HT,\conj}\into k^\nyg$.
  \item $\rnyg{R}$ pulls back to $\uspec(\HTp_{R/W(k)})$ along the locus $\spec(k)=k^\HT\into k^\nyg$.
  \item $\rnyg{R}$ pulls back to $\uspec_{\bgm}(\bigwedge^* L_{R/k}[-*])$ along the locus $\bgm=k^{\hodge}\into k^\nyg$.
    
\end{enumerate}

\end{lemma}
\begin{proof}
    These are immediate consequences of Lemma \ref{lem: functorialities for rees cons} and Construction \ref{cons: specialisation of rees stacks}.
    Indeed, note that the specializations in \ref{lem: functorialities for rees cons}  are implemented by base change of the corresponding derived algebras. The definitions of the stacks of Construction \ref{cons: specialisation of rees stacks} as functors corepresented by the base changes give the lemma.
\end{proof}

We next want to make contact with the quasi-regular semiperfectoid case of Bhatt's notes \cite[\S~5]{Bhatt22}. In characteristic $p$ these are the simplest class of ind-syntomic algebras with enough $p$-power roots to make de Rham type cohomological invariants discrete.

\begin{definition}\cite[Definition 8.8]{BhattMorrowScholze2019THHandPadicHodge}\label{definition of qrsp}
We say that an $\FF_p$-algebra $A$ is semiperfect if the Frobenius on $A$ is surjective. A semiperfect algebra is quasiregular if $L_{A/\FF_p}$ is a flat $A$-module concentrated in degree $1.$ 
\end{definition}

\begin{remark}
    Note that an $\FF_p$-algebra $A$ is quasiregular semiperfect if and only if it is quasiregular semiperfectoid in the sense of \cite[Definition 4.20]{BhattMorrowScholze2019THHandPadicHodge}. This is an easy consequence of the definition of quasiregular semiperfectoid along with the fact that the a priori $\ZZ_p$-algebra $A$ admits a morphism $\FF_p\to A$ along with the vanishing of $\Omega^1_{A/\FF_p}.$
\end{remark}

The key observation for us is the following. 

\begin{construction}\label{Rees construction}(Rees construction for $F^*\Prism_{A/W(k)}$ for quasiregular semiperfect $A$.)

If $A$ is quasiregular semiperfect then, as explained in \cite[\S~5.5.1]{Bhatt22}, its absolute prismatic cohomology $\Prism_A$ is discrete and carries a discrete Nygaard filtration $\fil_\nyg \Prism_{A}.$

If $A$ admits a map $k\to A$, then by \cite[Theorem 5.6.2]{APC}, there is a canonical isomorphism $$\fil_\nyg \Prism_{A}\simeq \fil_\nyg F^*\Prism_{A/W(k)},$$ and so its relative Nygaard filtered prismatic cohomology with respect to $W(k)$ is also discrete. 
Now since the Rees construction is functorial, there is also an induced isomorphism $$\rees(\fil_\nyg \Prism_{A})\simeq \rees (\fil_\nyg F^*\Prism_{A/W(k)})$$ of discrete graded $p$-complete $\wkut$-algebras. Thus we get a classically affine morphism $$\spf(\rees (\fil_\nyg F^*\Prism_{A/W(k)}))/\G_m\to k^\nyg$$ which agrees with the Rees stack of \cite[Definition 5.5.2 ]{Bhatt22}.
\end{construction}

\begin{lemma}\label{lem: uspf is gm}
Let $T$ be a discrete derived algebra in $\dalgk.$ Then the $k^\nyg$-stack $\uspf_{k^\nyg}(T)\to k^\nyg$ agrees with $\spf (T)/\GG_m$ where the $\GG_m$-action is induced on $T$ by identifying the heart of $\qcoh(k^\nyg)$ with $p$-complete graded $W\langle u,t\rangle/(ut-p)$-modules.
\end{lemma}
\begin{proof}
    We may assume we are in the classical situation i.e. of classical stacks over discrete commutative rings. The structure morphism $\wpt\to T$ is a graded morphism of $p$-complete graded rings. Thus the induced morphism $\spf T\to \spf \wpt$ intertwines the $\GG_m$-action on the source and target. By flat descent of affine morphisms of classical formal schemes, this gives us a classically affine morphism $\pi\colon \spf T/
    \GG_m\to \spf(\wpt)/\GG_m=k^\nyg.$ It suffices to identify the discrete algebra $\pi_*\cO_{\spf T/\GG_m}$. But identifying $\qcoh^\heartsuit(k^\nyg)$ with graded $\wpt$ modules, this functor just views $T$ as a graded $\wpt$-algebra hence proving the result. 
\end{proof}

\begin{remark}\label{rem: uspf is gm}
    It follows from Lemma \ref{lem: uspf is gm} that when $A$ is a quasiregular semiperfect algebra receiving a map from $k$, then the $\uspf_{k^\nyg}$-construction of Construction \ref{cons: rees stack of rnyg} agrees with the classical quotient stack of Construction \ref{Rees construction}. 
\end{remark}

\begin{recollection}[Adjunction for stacks over $k^\nyg.$]\label{rec: stacks adjunction}
    Let $f\colon X\to k^\nyg$ be any stack. Then $f_*\cO_X\in \dalgk$. 
    
    Let $A\in \dalgk$ be any derived algebra. Then there is an adjunction 
    $$\Maps_{\mathrm{Stacks}_{k^\nyg}}(X, \spf_{k^\nyg}(A))\simeq \Maps_{\dalgk}(A,f_*\cO_X).$$

    This adjunction can be deduced by the argument in the proof of \cite[Proposition 3.2]{MM24} where by taking colimits in the category of stacks over $k^\nyg$ one reduces to the case of an affine scheme over $k^\nyg$ and whence, by the affine diagonal of $k^\nyg$, it follows from the definition.
\end{recollection}

\begin{construction}[Construction of the comparison map.]\label{cons: nyg can map}
    Let $f\colon \spec(R)\to \spec(k)$ be a smooth map of  affine schemes and let $f^\nyg\colon R^\nyg\to k^\nyg$ be the associated map on the derived stacks. Then $f_*^\nyg\cO_{R^\nyg}=\fil_\nyg F^*\Prism_{R/W(k)}$ as derived algebras. This follows from the proof of \cite[Theorem 3.3.5]{Bhatt22} as commutative algebras wherein the proof as derived algebras can be deduced by noting that all the operations in the proof preserve derived algebras. 
    Now the functor $\fil_\nyg F^*\Prism_{-/W(k)}$ is left Kan extended from finitely generated polynomial $k$-algebras by \cite[Proposition 5.1.1]{APC}.
    Therefore by the adjunction of Recollection \ref{rec: stacks adjunction} and the universal property of left Kan extensions spelled out in Construction \ref{cons: can map}, we obtain a natural map $$\eta_R\colon R^\nyg\to \cR_{\nyg}(F^*\Prism_{R/W(k)})$$
of stacks over $k^\nyg.$
\end{construction}

\begin{theorem}[Bhatt-Lurie]\label{thm: eta is isomorphism when qrsp}
    When $A$ is a quasi-regular semiperfect algebra then the canonical map $\eta_A$ of Construction \ref{cons: nyg can map} is an isomorphism.
\end{theorem}
\begin{proof}
    After Construction \ref{Rees construction}, Lemma \ref{lem: uspf is gm} and Remark \ref{rem: uspf is gm}, this is \cite[Theorem 5.5.10 ]{Bhatt22}.\footnote{As written, it seems there is a slight gap in the justification of item (4) of \emph{loc.cit.} A different proof was communicated to the author by Bhargav Bhatt and will appear in \cite{Sah25Syn}.}
\end{proof}

\begin{remark}
    It should be possible to prove Theorem \ref{thm: intro Main theorem} directly from Theorem \ref{thm: eta is isomorphism when qrsp} by proving carefully that the functor $R\mapsto \rnyg{R}$ satisfies quasisyntomic descent. 
\end{remark}

\section{Proof of de Rham affineness of the Nygaard filtered prismatization in positive characteristic}\label{sec: de Rham affine proof}

In this section we prove that the canonical map $$\eta_R\colon R^\nyg\to \cR_\nyg(F^*\Prism_{R/W(k)})$$ of Construction \ref{cons: nyg can map} is an isomorphism. 

We begin first with two technical lemmas needed in the proof.

\begin{lemma}\label{lem: equiv of fpqc and p-fpqc}
  A morphism $R\to S$ of $p$-nilpotent animated rings is a faithfully flat morphism if and only if it is $p$-completely faithfully flat.   
\end{lemma}

\begin{proof}
   The forward direction is clear. For the converse note that if $A$ is an animated ring and $f\in \pi_0A$, then a module $M\in \Mod_A$ is $f$-completely flat if for any discrete $f$-torsion module $N$ one has that $M\otimes_A N$ is also discrete. Further it is faithfully flat iff $\pi_0M/f$ is faithfully flat over $\pi_0A/f$.  
   If $A$ itself is $f$-nilpotent i.e. $f^N=0$ in $\pi_0(A)$ then as recalled in \ref{flatness and t-exactness} $f$-complete flatness is equivalent to $M$ itself being flat over $A$ as $\Mod_A$ is entirely $f$-torsion. For faithfulness, we may lift faithful flatness of $\pi_0M/f$ over $\pi_0A/f$ along the nilpotent surjection $\pi_0A/f^N\to \pi_0A/f$ by \cite[\href{https://stacks.math.columbia.edu/tag/00HP}{Tag 00HP}]{stacks-project}.
\end{proof}

\begin{lemma}\label{lem: detecting fpqc covers}
    Let $f\colon X\to Y$ be a morphism of $p$-adic formal stacks. Assume that for any $p$-nilpotent point $\eta\colon \spec (T)\to Y$, the base change $\spec (T)\times_Y X$ is representable by an animated ring and $p$-completely faithfully flat over $\spec (T)$. Then $f\colon X\to Y$ is an effective epimorphism for the fpqc topology of $p$-nilpotent animated rings.
\end{lemma}

\begin{proof}
    Note that $f$ is an effective epimorphism for the fpqc topology if for any point $\eta\colon \spec (T)\to Y$ from a $p$-nilpotent animated ring $T$, there is, after an fpqc base change $\spec (T')\to \spec (T)$, a lift of $\eta$ to a point $\eta'\colon \spec (T')\to X.$ Assuming the hypothesis of the lemma, the morphism $\spec (T)\times_Y X\to X$ provides the requisite lift of $\eta$, noting that $\spec (T)\times_Y X\to \spec (T)$ is faithfully flat since it is $p$-completely faithfully flat by Lemma \ref{lem: equiv of fpqc and p-fpqc}.
\end{proof}

\begin{theorem}\label{thm: main thm}
 The map $\eta_R$ of Construction \ref{cons: nyg can map} is an isomorphism.   
\end{theorem}
\begin{proof}
    We will use Lemma \ref{lem: checking on a single cover} by the colimit preserving properties of both the source and target functors in question. For the source this follows from the explicit formula given by transmutation and for the target this follows from the main result of Appendix \ref{Appendix: colim derived nyg}.
    Thus we may just check for $A:=k[x]$ and consider the quasisyntomic cover $S=k[x^{1/p^\infty}].$ The target is quasiregular semiperfect (in fact perfectoid). Let $S^*$ be the Cech conerve of the canonical map $A\to S$. Note that each term of $S^*$ is quasiregular semiperfect by \cite[Lemma 4.30]{BhattMorrowScholze2019THHandPadicHodge}. Now since $(-)^\nyg$ preserves all colimits of $k$-algebras, for each $n\geq 1$, the tensor product $S^{\otimes_A^n}$ gets sent to the $n$-fold fiber product of $(S^{\nyg})$ over $A^\nyg$. In particular $(S^*)^\nyg$ is the Cech nerve of the maps $S^\nyg\to A^\nyg$ and we denote this by $(S^\nyg)^*.$ 
    
    By the colimit preserving properties of $\rnyg{-}$, a similar remark  holds for $\rnyg{S^*}$.
    
    Thus by Theorem \ref{thm: eta is isomorphism when qrsp}, we may conclude that the map of simplicial objects $$\eta_{S^*}\colon (S^*)^\nyg=(S^\nyg)^*\to \rnyg{S^*}=(\rnyg{S})^*$$ is an isomorphism.

    Now note that the functor $(-)^\nyg$ sends the quasisyntomic cover $A\to S$ to an fpqc epimorphism by \cite[Proposition 6.12.1.]{GM24}.

    Therefore to conclude we just need to show that the canonical map $\cR_\nyg(F^*\Prism_{S/W(k)})\to \cR_\nyg(F^*\Prism_{A/W(k)})$ is an fpqc epimorphism. We use the criterion of Lemma \ref{lem: detecting fpqc covers} to do so. First note that if $\spec(T)\to k^\nyg$ is a $p$-nilpotent ring mapping to $k^\nyg$, then by the affine diagonal of $k^\nyg$ there is an algebra $\cA\in \dalg^\cn(k^\nyg)$ whose relative spectrum is $\spec(T).$ Thus we have some graded $p$-complete animated algebra $A$ over the graded $p$-complete algebra $\wkut$ which has the additional property that it admits a morphism from $\rees(\fil_\nyg(F^*\Prism_{A/W(k)}))$, the latter considered as a graded $p$-complete derived algebra.
    Therefore to check the condition in Lemma \ref{lem: detecting fpqc covers} it will suffice to verify the stronger condition in the forthcoming Lemma \ref{lem: the criterion for fpqc}.
\end{proof}

We first need a general construction of Mathew--Mondal \cite{MM24}.

\begin{construction}\cite[Definition 2.1]{MM24}\label{cons: coconnective t-structure}
  Let $R$ be any (possibly non-connective) $E_1$-ring and let $\Mod_R$ be the $\infty$-category of $R$-modules. Then $\Mod_R$ admits a natural $t$-structure where the connective part is the smallest full subcategory generated under colimits and extensions by $R$ itself. The coconnective part is precisely those spectra which under the forgetful functor $\Mod_R\to \mathrm{Sp}$ lie in $\mathrm{Sp}_{\leq 0}$ i.e. the underlying spectrum is coconnective. 
  If $R\to S$ is a morphism of $E_1$-rings then the functor $-\otimes_R S$ is right $t$-exact i.e. preserves connective objects as it preserves colimits. Moreover this $t$-structure is compatible with filtered colimits in the sense of \cite[Definition 1.3.5.20]{HA}.
  When $R$ is connective then this is the usual $t$-structure on $\Mod_R$ studied in \cite[Proposition 7.1.1.13.]{HA}.
\end{construction}

\begin{lemma}\label{lem: connectivity by filtrations}
    Let $T$ be any (possibly non-connective) $E_\infty$-ring and $M$ a $T$-module. Suppose there is an increasing exhaustive $\mathbf{N}$-indexed filtration $\fil^{\geq 0}M$ on $M$ with each graded piece connective over $T$, then for any connective $E_\infty$-ring $B$ admitting a map of $E_\infty$-rings from $T$, the tensor product $M\otimes_T B$ has underlying spectrum connective. Moreover if each $\fil^n(M)$ is a coconnectively flat $T$-module (i.e. is left $t$-exact) then $M$ is coconnectively flat.
\end{lemma}
\begin{proof}
    We have cofiber sequences of $T$-modules

    $$\fil^0(M)=\gr_0(M)\to \fil^1(M)\to \gr_1(M)$$ whence by the long exact sequence for the homotopy groups of the $t$-structure in Construction \ref{cons: coconnective t-structure}, we observe that $\fil^1(M)$ is connective in $\Mod_T$. Inducting on $n$ along $$\fil^n(M)\to \fil^{n+1}(M)\to \gr_{n+1}(M)$$ we see that all the $\fil^n(M)$ are connective for $n\geq 0$. 

    Now we observe that by the last two sentences of Construction \ref{cons: coconnective t-structure} we have that $\fil^n(M)\otimes_T B$ are connective as $B$-modules for the usual $t$-structure on $\Mod_B$.
    Now the tensor product preserves colimits. As
    $M\otimes_T B=\colim_n \fil^n(M)\otimes_T B$, we learn that the left hand side is connective as connective modules are closed under colimits.

    It remains to show that if each of the $\fil^n(M)$ is coconnectively flat (i.e. left $t$-exact) then so is $M.$

    For this, let $K$ be any coconnective $T$-module. Again writing $K\otimes_T M=\colim_n \fil^n(M)\otimes_T K$, one notes that each module in the colimit is coconnective, and since the $t$-structure of Construction \ref{cons: coconnective t-structure} is compatible with filtered colimits, the module $K\otimes_T M$ remains coconnective.
\end{proof}
\begin{comment}
\begin{remark}\label{rem: exhaustive filtration by flat implies flat}
    In the setting of Lemma \ref{lem: connectivity by filtrations}, assume that each $\fil^i(M)$ is a coconnectively flat $T$-module i.e. is left $t$-exact. Then $M$ is also coconnectively flat since the $t$-structure of Construction \ref{cons: coconnective t-structure} is compatible with filtered colimits i.e. filtered colimits of coconnective modules remain coconnective.
\end{remark}
\end{comment}

\begin{lemma}\label{lem: the criterion for fpqc}
Keep the notation of the proof of Theorem \ref{thm: main thm} i.e. $A=k[x]$ and $S=k[x^{1/p^\infty}]$. Let $B$ be any $p$-complete connective graded $\wkut$-algebra admitting a morphism $\rees (\fil_\nyg F^*\Prism_{A/W(k)})\to B$. Then consider the $p$-complete graded $B$-algebra $$B'=B\widehat{\underline{\otimes}}_{\rees (\fil_\nyg F^*\Prism_{A/W(k)})} \rees (\fil_\nyg F^*\Prism_{S/W(k)})$$  where the $\widehat{\underline{\otimes}}$ is the graded $p$-completed tensor product of $p$-complete graded derived $\wkut$-algebras. 

After passing to underlying algebras, the $B$-algebra $B'$ is connective and $p$-completely faithfully flat.
\end{lemma}

\begin{proof}
Assume that we can verify the algebra is connective. Then we need to prove $p$-complete flatness of $B'$ over $B.$ This amounts to showing that $B'\qq p$ is a faithfully flat module over $B\qq p$. Via Corollary \ref{faithful on constructible}, it will suffice to check that the pullback of the $B\qq p$ algebra $B'\qq p$ is  faithfully flat when restricted to $k^\crys, k^\HT$ and $k^\hodge.$ This is in the forthcoming Corollary \ref{cor: pcomp flat on prism}. Note here that $k^\crys$ is the Frobenius pullback of $k^\Prism$ as follows from the explicit constructions of the corresponding ring stacks, recalled in Constructions \ref{const prism} and \ref{const: cryst}. Thus it remains to check that $B'$ is connective as a $B$-algebra. 

For this set $T:=\rees(\fil_\nyg F^*\Prism_{A/W(k)})$ and $Q:=\rees(\fil_\nyg F^*\Prism_{S/W(k)})$ and  $C=\wkut.$ Following a suggestion of Bhatt, we show that the $T$-module $Q$ has an increasing exhaustive $\N$-indexed filtration via the Bar resolution of $Q$ taken as a $T$-module in the $\infty$-category of graded $p$-complete $C$-modules. We have $$\Barc_{T}(T,Q):=\bigg(\simp { T\widehat{\underline{\otimes}}_C Q}{ T\widehat{\underline{\otimes}}_CT\widehat{\underline{\otimes}}_C Q}{T\widehat{\underline{\otimes}}_C T\widehat{\underline{\otimes}}_C T\widehat{\underline{\otimes}}_C Q }\bigg)
$$ which is a split simplicial object by \cite[Example 4.7.2.7]{HA} augmenting  on $Q$ via the action map.
In other words $$|\Barc_{T}(T,Q)|=Q.$$ 
Since the forgetful functor $\Gr\Mod_C^{\pcomp}\to \Mod_C^\pcomp$ is colimit preserving and symmetric monoidal (here the first $C$ is taken with its grading), we see that, even in $\Mod_C^\pcomp$, the bar complex $\Barc_{T}(T,Q)$ is a resolution of $Q.$ 
Via the $\infty$-categorical Dold-Kan correspondence in \cite[\S~1.2.4]{HA}, we get an increasing exhaustive $\N$-indexed filtration $\fil^\bullet Q$ on $Q$ with graded given by $$\gr_nQ=T^{\widehat{\otimes}n+1}\widehat{\otimes}_C Q[n].$$
By the first part of Lemma \ref{lem: connectivity by filtrations}, it will suffice to check that $T^{\widehat{\otimes}n+1}\widehat{\otimes}_C Q[n]$ is connective over $T$ where the $T$-module structure is determined by the leftmost factor. 
By Construction \ref{cons: coconnective t-structure} it suffices to prove that $T^{\widehat{\otimes}n}\otimes_C Q[n]$ has underlying spectrum connective. Since $Q$ is connective, it suffices to check that $T^{\widehat{\otimes}n}[n]$ is connective. The Nygaard filtration on a smooth algebra is complete (e.g. by the proof of \cite[Theorem 3.3.5]{Bhatt22} or by \cite[Corollary 5.2.11]{APC}) whence $T$ is $t$-complete as a $C$-module. Thus by Lemma \ref{connective mod f}, it suffices to check $T\qq t$ is $-1$-connective. But $T\qq t=\rees(\fil_{\mathrm{conj}}\HTp_{A/W(k)})$ and the latter is clearly $-1$-connective e.g. by \cite[Remark 4.1.12]{APC}.

\end{proof}

We first use a consequence of \cite[Theorem 7.17]{APCII}.

\begin{lemma}\label{lem: pcomp algebra over prism is pcomp}
   Keep the notation of the proof of Theorem \ref{thm: main thm} i.e. $A=k[x]$ and $S=k[x^{1/p^\infty}]$. Let $B$ be a $p$-complete animated $\Prism_{A/W(k)}$-algebra. Then the $p$-complete $B$ algebra $B':= B\widehat{\otimes}_{\Prism_{A/W(k)}}\Prism_{S/W(k)}$ is connective and $p$-completely flat over $B.$
\end{lemma}
\begin{proof}
    This is in the proof of \cite[Theorem 7.17]{APCII}. We sketch their proof for the convenience of the reader. By Lemma \ref{lem: connectivity by filtrations} %and Remark \ref{rem: exhaustive filtration by flat implies flat} 
    it will suffice to show that $\Prism_{S/W(k)}\qq p\simeq \HTp_{S/W(k)}$ admits an exhaustive increasing $\N$-indexed filtration with graded free over $\Prism_{A/W(k)}\qq p\simeq \HTp_{A/W(k)}.$ 
    Both $S$ and $A$ admit $p$-completely flat Frobenius lifts (or $\delta$-lifts) $\widetilde{S}:=W(k)[x^{1/p^\infty}]^\wedge_p$ and $\widetilde{A}:=W(k)[x]^\wedge_p$ to $W(k)$, where the Frobenius is just $x\mapsto x^p$. The $\widetilde{A}$ canonically splits the Hodge-Tate gerbe $A^\HT\to \spec(A)$ recalled in Construction \ref{const HT}, and gives an isomorphism of derived $k$-algebras $$\HTp_{A/W(k)}\simeq \bigoplus_{i\geq 0}\Omega^i_{A/k}[-i],$$ where the Breuil-Kisin twists have been canonically trivialised in characteristic $p.$ The lift $\widetilde{S}$ induces an isomorphism $S^\HT\simeq S$ whence the map $\HTp_{A/W(k)}\to \HTp_{S/W(k)}$ is carried over to the map $$\bigoplus_{i\geq 0}\Omega^i_{A/k}[-i]\to A\to S$$ where the left map is the augmentation from the de Rham complex. Since $A\to S$ is clearly faithfully flat, it suffices to show that $A$ admits an $\N$-indexed exhaustive increasing filtration as a $\bigoplus_{i\geq 0}\Omega^i_{A/k}[-i]$-module so that the graded are all free over $\bigoplus_{i\geq 0}\Omega^i_{A/k}[-i]$. This is a shift of the calculation done in \cite[Lemma 8.6]{BS19}, which we recall in our forthcoming Lemma \ref{lem: filtrations on de rham complex} (since the case of the Hodge complex is useful in the sequel).
\end{proof}

\begin{remark}\label{rem: semi-free resolutions}
    The next lemma builds semi-free resolutions over the de Rham complex considered as a DG-algebra. We refer the reader to \cite[\S~13]{Drin08} for more on this.
\end{remark}

\begin{lemma}\label{lem: filtrations on de rham complex}
Keep the notation of the proof of Theorem \ref{thm: main thm} i.e. $A=k[x]$. Then the $\Omega:=\bigoplus_{i\geq 0}\Omega^i_{A/k}[-i]$-module $A$ admits an increasing exhaustive $\N$-indexed filtration such that the graded pieces are all free over $\bigoplus_{i\geq 0}\Omega^i_{A/k}[-i]$. The same remains true when we consider $A$ as a module over $\Omega^{\hodge}:=\bigoplus_{i\geq 0}\Omega^i_{A/k}[-i]$ as a complex with zero differentials . 
\end{lemma}

\begin{proof}
    We have $\Omega=A\bigoplus \Omega^1_{A/k}[-1]$ with differential given by the de Rham differential. We have the augmentation $\Omega\to A$ and the fiber is given by $\Omega^1_{A/k}[-1]\simeq Adx[-1]$. Thus there is again a canonical morphism $\Omega[-1]\to \Omega^1_{A/k}[-1]$ with fiber $\Omega^1_{A/k}[-2]$ and so on. Thus we have built a semi-free resolution of the DG $\Omega$-module $A$ as $$(\ldots \Omega[-3]\to \Omega[-2]\to \Omega[-1]\to \Omega)\simeq A.$$
    We may endow the chain complex above with the stupid increasing filtration (which is always exhaustive) which then endows $A$ with a filtration $\fil^\bullet_{\mathrm{naive}}A$ which has the property that $\gr^i_{\mathrm{naive}}A\simeq \Omega$. So we conclude. 

    For the case of $\Omega^\hodge$ we just observe that the above resolution did not use the differential on $\Omega$ except for building the chain maps. Thus $A$ admits an increasing exhaustive $\N$-indexed filtration as an $\Omega^\hodge$ module so that the graded are again given by $\Omega^\hodge.$
 
\end{proof}

\begin{corollary}\label{cor: pcomp flat on prism} We have the following flatness results:
\begin{enumerate}
    \item The restriction of $B'$ as a $B$ algebra to the locus $k^\Prism$ of Construction \ref{const prism} is $p$-completely faithfully flat.

    \item  The restriction of $B'$ as a $B$ algebra to the locus $k^\HT$ of Construction \ref{const HT} is  faithfully flat.

    \item  The restriction of $B'$ as a $B$ algebra to the locus $k^\hodge$ of Construction \ref{const: hodge} is  faithfully flat.
\end{enumerate}
    
\end{corollary}
\begin{proof}
\begin{enumerate}
    \item After Lemma \ref{lem: pcomp algebra over prism is pcomp}, this follows from Lemma \ref{lem: connectivity by filtrations}. %and Remark \ref{rem: exhaustive filtration by flat implies flat}.
    \item  This is a consequence of the proof of Lemma \ref{lem: pcomp algebra over prism is pcomp}.
    \item   Using the reasoning of Lemma \ref{lem: pcomp algebra over prism is pcomp}, this follows from the second assertion of Lemma \ref{lem: filtrations on de rham complex}.
\end{enumerate}
   
\end{proof}

We will now record some corollaries of Theorem \ref{thm: main thm}.

\begin{corollary}\label{cor: substacks iso}[De Rham affineness of substacks of $R^\nyg.$]
    For any animated $k$-algebra $R$ there is an isomorphism, natural in $R$
    \begin{enumerate}
        \item $ (\eta_{R})_{t\neq 0}\colon (R/W)^\crys \simeq \uspf(F^*\Prism_{R/W(k)})$ over $\spf(W(k))=k^\crys.$
        \item $ (\eta_{R})_{u\neq 0}\colon R^\Prism\simeq  \uspf(\Prism_{R/W(k)})$ over $\spf(W(k))=k^\Prism.$

        \item $(\eta_{R})_{u=0}\colon (R/k)^{\dR,+}\simeq \rhge{R}$ over $\agm=k^{\dR,+}.$

        \item $(\eta_{R})_{t=0}\colon (R/W(k))^{\HT,\conj}\simeq \rconj{R}$ over $\agm=k^{\HT,\conj}.$
        \item $(\eta_{R})_{t=0, u\neq 0}\colon R^{\HT}\simeq \uspec(\HTp_{R/W(k)})$ as stacks over $\spec(k)=k^{\HT}$

        \item $(\eta_{R})_{t=0, u=0} \colon (R/k)^{\hodge}\simeq \uspec_{\bgm}(\bigwedge^* L_{R/k}[-*])$ as stacks over $\bgm\simeq k^{\hodge}.$
    \end{enumerate}
\end{corollary}

\begin{proof}
    These are all obtained by the restriction of $\eta_R$ to the corresponding loci of $k^\nyg.$
\end{proof}

\appendix

\section{Testing faithful flatness on a constructible stratification}\label{appendix on testing flatness}
\renewcommand{\thetheorem}{\Alph{section}.\arabic{theorem}}
In this appendix we prove some of the faithful flatness lemmas needed in the main body of the note. Fix an animated ring $A$ and $f\in \pi_0(A).$  We are interested in testing $t$-exactness for modules $M\in \mathrm{Mod}(A)$ on the `constructible' stratification of the derived  scheme $\spec A$ given by $\spec A[1/f]$ and the underlying topological space of $\spec A/f$.  This will be obtained by an analysis of the category of complexes on both these loci. The main results are recorded in Corollaries \ref{flat on constructible} and \ref{faithful on constructible}.

We begin with some general remarks.

\begin{remark}(Recollections on flatness over animated rings)\label{flatness and t-exactness} Let $A$ be an animated ring and $M\in \Mod_A$. Then recall from \cite[Definition 7.2.2.10]{HA} that $M$  is called flat if $\pi_0M$ is a flat $\pi_0(A)$ module and the natural action map $\pi_*(A)\otimes_{\pi_0(A)}\pi_0(M)\to \pi_*(M)$ is an isomorphism of $\ZZ$-graded abelian groups (in particular, $M$ is connective). By the derived variant of Lazard's theorem \cite[Theorem 7.2.2.15]{HA}, when $M$ is already known to be connective, this condition is equivalent to the following two (easier to work with) and equivalent conditions
    \begin{enumerate}
        \item The functor $M\otimes_A-$ is left $t$-exact on $\Mod_A$ i.e. sends coconnective modules to coconnective modules,
        \item The functor $M\otimes_A -$ preserves discrete $A$-modules. 
    \end{enumerate}

    In the sequel we will use flat and left $t$-exact interchangeably for connective modules.
\end{remark}

\begin{remark}(Recollection on \emph{faithful flatness} over an animated ring.)\label{faithful on pi0} Following \cite[Definition 5.2]{Lur11b} we say that a module $M$ over an animated ring $A$ is faithfully flat if it is flat in the sense of Remark \ref{flatness and t-exactness} and $\pi_0M$ is faithfully flat over $\pi_0A$.
    
\end{remark}

\begin{remark} \label{loc coh}(Recollections on local cohomology.)
    Let $A$ be an animated ring and $I\subset \pi_0(A)$ a finitely generated ideal. Then recall from \cite[\S~4.1]{Lur11} that the category $\Mod_A^{I\text{-}\mathrm{tors}}$ of $I$-torsion modules is a colocalizing subcategory of $\Mod_A$ consisting of those modules whose homotopy groups are endowed with a locally nilpotent action of $I$. The right adjoint to the inclusion $\Mod_A^{I\text{-}\mathrm{tors}}\hookrightarrow \Mod_A$ is denoted by $\Gamma_I(-)\colon \Mod_A \to \Mod_A^{I\text{-}\mathrm{tors}}$  and agrees with Grothendieck's theory of local cohomology in the Noetherian case. We are here interested in the following properties of local cohomology when $I=(f) $ is a principal ideal. Fix a module $M\in \Mod_A.$

    \begin{enumerate}
        \item There is a cofiber sequence  
    $$ \Gamma_{(f)}(M)\to M\to M[1/f]$$
  which is a sequence of coconnective modules whenever $M$ is coconnective. Indeed, this follows from \cite[Example 4.1.14.]{Lur11} and \cite[Proposition 4.1.18.]{Lur11} respectively.

 \item The canonical map $\Gamma_{(f)}(M/\!\!/f)\to M/\!\!/f $ is an isomorphism. This can be seen as the cofiber is $(M/\!\!/f)[1/f]$ by $(1)$ and is clearly $0$ since $f$ acts both null-homotopically and invertibly on it. More conceptually, $M/\!\!/f$ has $f$-torsion homotopy groups, and therefore is in $ \Mod_A^{I\text{-}\mathrm{tors}}$ and so $\Gamma_{(f)}(-)$ is the identity functor. 

 \item There is an equivalence $\Gamma_{(f)}M=\Gamma_{(f)}(A)\tensor_A M$. In particular for any $N\in \Mod_A$ we have $\Gamma_{(f)}(M)\tensor_A N=\Gamma_{(f)}(M\otimes_A N)=M\otimes_A \Gamma_{(f)}(N)$ by \cite[Proposition 4.1.12]{Lur11}. In particular $\Gamma_{(f)}(M)\tensor_A N\in \Mod^{I\text{-}\mathrm{tors}}_A$ always (see also \cite[Corollary 4.1.7]{Lur11} for a more general statement).

    \end{enumerate}
\end{remark}

We will now use the above observations to prove a constructible criterion for faithful flatness. This statement should certainly be well known but we don't know a reference.

\begin{lemma}\label{left t-exact constructible}
    Let $M\in \mathrm{Mod}_A$. Then the following conditions are equivalent
    \begin{enumerate}
        \item $M\otimes_A-$ is left $t$-exact,
        \item $M[1/f]\otimes_{A[1/f]}-$ and $M/\!\!/f\otimes_{A/\!\!/f}-$ are left $t$-exact on $\mathrm{Mod}_{A[1/f]}$ and $\Mod_{A/\!\!/f}$ respectively
    \end{enumerate}
    
\end{lemma}
\begin{proof}
The forward direction is immediate since $t$-exactness for connective $E_\infty$-rings is a property of the underlying spectra. 

So let us prove the other direction. 
 We will show that if $N\in (\mathrm{Mod}_A)_{\leq 0}$ then $M\otimes_A N\in (\Mod_A)_{\leq 0}$. 
 
Using point $(1)$ of Remark \ref{loc coh}, there is a cofiber sequence of coconnective $A$-modules
 
 $$
     \Gamma_{(f)}(N)\to N\to N[1/f]
 $$
 
 Tensoring with $-\otimes_A M$ we get 
 
 \begin{equation}\label{orthogonal cofiber tensor}
     \Gamma_{(f)}(N)\tensor_A M\to N\tensor_A M\to N[1/f]\tensor_A M.
 \end{equation}
 
Observe that if the left and right terms of \ref{orthogonal cofiber tensor} are coconnective then so is the middle term and we would be done.

 First note that because $N[1/f]\otimes_A M=N[1/f]\otimes_{A[1/f]} M[1/f]$, we learn that the rightmost term of \ref{orthogonal cofiber tensor} is in fact coconnective by the hypothesis on $M.$

 We are left to understand $\Gamma_{(f)}(N)\tensor_A M$. 
 For this, first consider the canonical cofiber sequence $$
     M\xrightarrow{f}M\to M/\!\!/f
$$ which we tensor with $\Gamma_{(f)}(N)$ to get 

 \begin{equation}\label{local cofiber}
     \Gamma_{(f)}(N)\tensor_A M\xrightarrow{f}  \Gamma_{(f)}(N)\tensor_A M \to  \Gamma_{(f)}(N)\tensor_A M/\!\!/f 
 .\end{equation}

Note that the rightmost term is the cofiber of multiplication by $f$ on $\Gamma_{(f)}(N)\tensor_A M.$

To control the terms in \ref{local cofiber}, we use the colocalizing properties of local cohomology as explained in Remark \ref{loc coh}. Specifically by point $(3)$ of Remark \ref{loc coh} we have $\Gamma_{(f)}(N)\tensor_A M\qq f=N\tensor_A \Gamma_{(f)}(M/\!\!/f)$ and by point $(2)$ of Remark \ref{loc coh} we have $N\tensor_A \Gamma_{(f)}(M/\!\!/f)=N\tensor_A M/\!\!/f.$

Now $N\tensor_A M/\!\!/f=N/\!\!/f\tensor_{A/\!\!/f}M/\!\!/f$ and since $N/\!\!/f\in (\Mod_{A/\!\!/f})_{\leq 1}$ we learn that $N/\!\!/f\tensor_{A/\!\!/f}M/\!\!/f\in (\Mod_{A/\!\!/f})_{\leq 1}$ by the $t$-exactness hypothesis on $M/\!\!/f$ as an $A/\!\!/f$-module. 

All in all \begin{equation}\label{1-coconnective}
     \Gamma_{(f)}(N)\tensor_A M/\!\!/f \in (\Mod_A)_{\leq 1}.
\end{equation}

Thus it remains to control the first two terms of \ref{local cofiber}. Note that by point $(3)$ of Remark \ref{loc coh} the homotopy groups of $\Gamma_{(f)}(N)\tensor_A M$ are $(f)$-torsion. Then taking the long exact sequence associated to \ref{local cofiber} the relevant terms for $i\geq 1$ are $$\ldots\to 0=\pi_{i+1}(\Gamma_{(f)}(N)\tensor_A M/\!\!/f)\to \pi_i(\Gamma_{(f)}(N)\tensor_A M) \xrightarrow{f} \pi_i(\Gamma_{(f)}(N)\tensor_A M)\to \ldots.$$

Now since $\pi_i(\Gamma_{(f)}(N)\tensor_A M)$ is entirely $(f)$-torsion for all $i\in \ZZ$, we observe that if it is non-zero for $i\geq 1$ then $\pi_{i+1}(\Gamma_{(f)}(N)\tensor_A M/\!\!/f)\neq 0$ which contradicts \ref{1-coconnective}.

\end{proof}

\begin{corollary}\label{flat on constructible}
    Let $M\in (\Mod_A)_{\geq 0}$. Then $M$ is $A$-flat if and only if $M/\!\!/f$ is $A/\!\!/f$-flat and $M[1/f]$ is $A[1/f]$-flat.
\end{corollary}
\begin{proof}
    This follows by combining Lemma \ref{left t-exact constructible} and Remark \ref{flatness and t-exactness}.
\end{proof}

\begin{remark}\label{discrete when discrete}
When $A$ is discrete, then derived Lazard's theorem as recalled in Remark \ref{flatness and t-exactness} implies that an $M\in (\Mod_A)_{\geq 0}$ satisfying the conditions in Lemma \ref{left t-exact constructible} is  itself discrete.
\end{remark}

We will now end this section with the needed faithfulness statement. 

\begin{corollary}\label{faithful on constructible}
    Let $M\in (\Mod_A)_{\geq 0}$ be an $A$-module. Then $M$ is faithfully flat if and only if $M/\!\!/f$ is $A/\!\!/f$ faithfully flat and $M[1/f]$ is $A[1/f]$-faithfully flat.
    \end{corollary}

\begin{proof}
    Flatness has been established in Corollary \ref{flat on constructible}. We just need to check faithfulness. By Remark \ref{faithful on pi0} we are reduced to checking that $\pi_0M$ is a faithfully flat $\pi_0A$ module if and only if $\pi_0(M/\!\!/f)=\pi_0(M)/f$ and $\pi_0(M[1/f])=\pi_{0}(M)[1/f]$ are faithfully flat $\pi_0(A/\!\!/f)=\pi_0(A)/f$ and $\pi_0(A[1/f])=\pi_0(A)[1/f]$-modules respectively. 

    Thus we reduce to the discrete case. Note that $\pi_0(M),\pi_0(M)/f$ and $\pi_0(M)[1/f]$ are discrete flat modules over their respective algebras. It suffices to check that they have non-zero fibers. But then this follows from \cite[\href{https:/\!\!/stacks.math.columbia.edu/tag/00HP}{Tag 00HP}]{stacks-project}.  
\end{proof}

\begin{remark}\label{connective on constructible}
    One may wonder if connectivity can also be tested on constructible stratifications i.e. if $M\in \Mod(A)$ is an arbitrary module such that $M/\!\!/f\in \Mod(A/\!\!/f)$ is connective and $M[1/f]\in \Mod(A[1/f])$ is connective then is $M$ connective? The answer is no as the example $\QQ_p/\ZZ_p[-1]=\fib(\ZZ_p\to \QQ_p)$ as an object of $\Mod(\ZZ_p)$ shows. From the fiber sequence one sees that mod $p$ one has $\QQ_p/\ZZ_p[-1]/\!\!/p=\FF_p$ and after inverting $p$ one has $\QQ_p/\ZZ_p[-1][1/p]=0$. In fact this is the worst that one can do as the next lemma shows.
\end{remark}

\begin{lemma}
    Let $M$ be an $A$ module which is connective mod $f$ and after inverting $f$ as in Remark \ref{connective on constructible}, then $M$ is $(-1)$-connective and $\pi_{-1}(M)$ is $f$-divisible and $f$-torsion. In particular if $\pi_{-1}M$ is (derived) $f$-complete, then $M$ is connective. \footnote{The author discovered this lemma with Ryo Ishizuka.}
\end{lemma}

\begin{proof}
    First note that $\pi_i(M/\!\!/f)$ vanishes for $i<0$ and so we see from the cofiber sequence $M\to M\to M/\!\!/f$ that $\pi_i(M)$ is $f$-divisible for all $i<0$ and multiplication by $f$ is invertible for $i<-1.$ On the other hand $\pi_{i}(M[1/f])=\pi_{i}(M)[1/f]=0$ for all $i<0$ implies that $\pi_i(M)=0$ for $i<-1.$ 

    Now consider the exact triangle $\Gamma_{(f)}(M)\to M\to M[1/f]$ and the connectivity of $M[1/f]$ implies that there is a surjection $\pi_{-1}(\Gamma_{(f)}(M))\to \pi_{-1}(M)$ which then proves the result as the homotopy groups of the source are entirely $f$-torsion.
\end{proof}

\section{Testing connectivity modulo $J$ for $J$-complete objects}\label{appendix: connectivity mod J}

Let $A$ be an animated ring and $J\subset \pi_0A$ a finitely generated ideal.  Suppose $M\in \Mod^{\jcomp}_A$ is such that after choosing generators $J=(f_1,\ldots, f_r)$ the quotient $M/\!\!/(f_1,\ldots,f_r)$ is connective, then is $M$ connective? As we will see below the answer is affirmative and should be well known, but the author couldn't find a reference.\footnote{At the time of writing there was no reference. But since then the author has communicated the lemma to the Stacks Project, and a variant of the argument can be found in the relevant sections here \cite[\href{https://stacks.math.columbia.edu/tag/0H83}{Tag 0H83}]{stacks-project}. We record the argument here for animated rings. } Note that the result in Corollary \ref{connectivity modulo J} is stronger than derived Nakayama for discrete rings. 

\begin{lemma}\label{classical nakayam}
    Let $A$ be a discrete ring, $f\in A$ and $M$ a discrete (derived) $f$-complete $A$-module. If $M/fM=0$ then $M=0.$
\end{lemma}
\begin{proof}
     Indeed, the completeness condition says that the derived limit $T(M,f)\colon=\lim(\ldots \xrightarrow{f}M\xrightarrow{f}M\xrightarrow{f} M)=0.$ In particular $\pi_0T(M,f)$ vanishes, but that is just the classical limit of multiplication by $f$ on $M$, which is zero if and only if $M=0$ as $M/fM=0$.
\end{proof}

\begin{lemma}\label{connective mod f}
    Let $A$ be an animated ring and $f\in \pi_0(A)$. Assume $M\in \Mod_A$ is $(f)$-complete. Suppose $M\qq f$ is connective. Then $M$ is connective.
\end{lemma}
\begin{proof}
    We have a cofiber sequence $$M\qq f^{n-1}\to M\qq f^n\to M\qq f$$ and by induction on $n$ we learn $M \qq f^n$ is connective for all $n>0.$  But note that $M\simeq \llim_n M\qq f^n\in (\Mod_A)_{\geq -1}.$ 
    The cofiber sequence $M\xrightarrow{f} M\to M\qq f$ shows that $\pi_{-1}(M)$ is $f$-divisible. But it is also derived $(f)$-complete as a $\pi_0(A)$-module by \cite[Theorem 4.2.13]{Lur11} and so $0$ by Lemma \ref{classical nakayam}.
\end{proof}

\begin{corollary}\label{connectivity modulo J}
    Let $M$ be a $J$-complete module such that after choosing generators $J=(f_1,\ldots, f_r)$ the quotient $M\qq (f_1,\ldots, f_r)$ is connective. Then $M$ is connective.
\end{corollary}
    
\begin{proof}
    Note that since $M$ is $J$-complete, by induction on $i$ we see that the sequence of modules $M\qq (f_1,\ldots, f_i)$ is $J$-complete being the iterated cofiber of multiplication by $f_i$ on $M\qq(f_1,\ldots, f_{i-1}).$ In particular, by \cite[Corollary 4.2.12]{Lur11} the module $M\qq(f_1,\ldots, f_{i-1})$ is $f_i$-complete.

    By descending induction on $i$, we conclude by Lemma \ref{connective mod f} that $M\qq(f_1,\ldots, f_{i-1})$ is connective. Thus $M$ is connective.
\end{proof}

\section{Colimit preservation and K\"unneth for derived Nygaard filtered prismatic cohomology}\label{Appendix: colim derived nyg}

We will prove a K\"unneth formula for derived Nygaard filtered prismatic cohomology. Our proof follows \cite[\S~3.5]{AnschutzLeBras2023PrismaticDieudonne} along with the constructible stratification argument already used in the main text. 

In this section we let $(A,I)$ be a bounded prism and set $\oA=A/I$. Let $R$ and $S$ be $p$-completely smooth (discrete) $\oA$-algebras. Our goal is to prove that the canonical morphism of filtered derived algebras over the filtered ring $\fil_I A$

\begin{equation}\label{can mor}
   \mathrm{KM}\colon \fil_{\nyg}F^*\Prism_{R/A}\widehat{\otimes}_{\fil_IA} \fil_\nyg F^*\Prism_{S/A}\to \fil_{\nyg} F^*\Prism_{R\widehat{\otimes}_{\oA}S/A}
\end{equation}

is an equivalence.

The functor $\fil_\nyg F^*\Prism_{-/A}$ is left Kan extended from $p$-completely smooth $\oA$-algebras. Therefore, if \ref{can mor} is an isomorphism, it follows formally from \cite[Construction 2.1]{BhattMorrowScholze2019THHandPadicHodge} that the functor $\fil_\nyg F^*\Prism_{-/A}$ preserves all colimits. 
In particular we have 

\begin{theorem}\label{general kunneth}
    Let $R\leftarrow T\rightarrow S$ be a span of animated $p$-complete $\oA$-algebras. Then the canonical morphism of filtered derived $(p,I)$-complete $\fil_IA$ algebras $$\fil_{\nyg}F^*\Prism_{R/A}\widehat{\otimes}_{\fil_\nyg F^*\Prism_{T/A}} \fil_\nyg F^*\Prism_{S/A}\to \fil_{\nyg} F^*\Prism_{(R\widehat{\otimes}_{T}S)/A}$$ is an isomorphism.
\end{theorem}

\begin{proof}
    Follows from Lemma \ref{can kunneth}.
\end{proof}

\begin{lemma}\label{hodgecohomology kunneth}
    The functor which sends a $p$-complete animated $\oA$-algebra $R$ to $\bigwedge^* L\widehat{\Omega}_{R/(\oA)}[-*]$ commutes with all colimits. In particular for a span $R\leftarrow T\rightarrow S$ the canonical map $$\mathrm{KM}\colon \bigwedge^* L\widehat{\Omega}_{R/(\oA)}[-*]\widehat{\otimes}_{\bigwedge^* L\widehat{\Omega}_{T/(\oA)}[-*]}\bigwedge^* L\widehat{\Omega}_{S/(\oA)}[-*]\to  \bigwedge^* L\widehat{\Omega}_{R\widehat{\otimes}_{T}S/(\oA)}[-*]$$ is an isomorphism.
\end{lemma}

\begin{proof}
    It suffices to check that the functor which sends a $p$-completely smooth $\oA$-algebra $R\mapsto \bigwedge^*\widehat{\Omega}_{R/(\oA)}[-*]$ commutes with finite co-products. But this is classical. 
\end{proof}

\begin{lemma}\label{HT commutes with colimits}
    The functor which sends a $p$-complete animated $\oA$-algebra $R$ to the $p$-complete derived $\oA$-algebra $\HTp_{R/A}$ commutes with all colimits. In particular for a span $R\leftarrow T\rightarrow S$ the canonical morphism $$\mathrm{KM}\colon \HTp_{R/A}\widehat{\otimes}_{\HTp_{T/A}} \HTp_{S/A}\to \HTp_{R\widehat{\otimes}_T S/A}$$ is an isomorphism.
\end{lemma}
\begin{proof}
    The Hodge-Tate complex $\HTp_{R/A}$ has a functorial exhaustive increasing $\mathbf{N}$-indexed conjugate filtration $\fil^\bullet_{\conj}\HTp_{R/A}$ and graded given by $$\gr_n \HTp_{R/A}=L\widehat{\Omega}^n_{R/(\oA)}\{-n\}[-n]$$ %and the right hand side commutes with all colimits. It follows by induction on $n$ that $R\mapsto \fil^n \HTp_{R/A}$ commutes with all colimits and so does $R\mapsto \HTp_{R/A}=\colim_n \fil^n \HTp_{R/A}.$
    It will suffice to show that for $p$-completely smooth $\oA$-algebras $R$ and $S$, the canonical map of filtered algebras

    $$\fil^\bullet_{\conj} \HTp_{R/A}\widehat{\otimes}_{\oA} \fil_\conj^\bullet \HTp_{S/A}\to \fil^\bullet_{\conj}\HTp_{R\widehat{\otimes}_{\oA} S/A} $$ is an isomorphism. This follows from Lemma \ref{hodgecohomology kunneth}.

\end{proof}

\begin{lemma}\label{kunneth for prism}
    The functor which sends a $p$-complete animated $\oA$-algebra $R$ to the $(p,I)$-complete derived $A$-algebra $\Prism_{R/A}$ commutes with all colimits. In particular for a span $R\leftarrow T\rightarrow S$ the canonical morphism $$\mathrm{KM}\colon \Prism_{R/A}\widehat{\otimes}_{\Prism_{T/A}} \Prism_{S/A}\to \Prism_{R\widehat{\otimes}_T S/A}$$ is an isomorphism. Similarly the Frobenius twisted version 
    $$\mathrm{KM}\colon F^*\Prism_{R/A}\widehat{\otimes}_{ F^*\Prism_{T/A}} F^* \Prism_{S/A}\to  F^*\Prism_{R\widehat{\otimes}_T S/A}$$ is also an isomorphism.
\end{lemma}
\begin{proof}
By $I$-completeness of the modules in question this may be checked mod $I$ where this follows from \ref{HT commutes with colimits}. The second statement follows because $A\widehat{\otimes}_{A,\varphi}(-)$ is symmetric monoidal and preserves all colimits.
\end{proof}

The idea of using the Rees construction in the next argument was suggested to us by Bhargav Bhatt.

\begin{lemma}\label{can kunneth}
    The morphism in \ref{can mor} is an isomorphism. 
\end{lemma}
\begin{proof}
We may prove this at the level of modules. The modules live in the category $\Mod_{\fil_I A}(\fil\Mod^{(p,I)\text{-comp}}_A)$. We may use the Rees equivalence, which is symmetric monoidal, to transport these to a statement about $\Gr\Mod^{(p,I)\text{-comp}}_{\mathrm{Rees}(\fil_IA)}.$

Unwinding definitions we have a morphism of modules 
$$\mathrm{KM}\colon \rees(\fil_{\nyg}F^*\Prism_{R/A})\widehat{\otimes}_{\rees(\fil_IA)} \rees(\fil_\nyg F^*\Prism_{S/A})\to \rees(\fil_{\nyg} F^*\Prism_{R\widehat{\otimes}_{\oA}S/A}).$$

After a faithfully flat base change by \cite[Lemma 3.7]{BS19}  we may assume that the prism is orientable. So we have $I=(d)$ for some generator $d.$
Then $\rees(\fil_I(A))=A[u,t]/(ut-d)$ as a graded ring. We may then view the algebras in question as $(p,d)$-complete algebras.
We set $F=\fib(\mathrm{KM})$ and we will show that $F$ vanishes.

When $u=0,t=0$ the fiber vanishes by Lemma \ref{hodgecohomology kunneth}. When $u\neq 0$ and $t=0$ then the fiber vanishes by Lemma \ref{HT commutes with colimits}. When $t\neq 0$, the fiber vanishes by Lemma \ref{kunneth for prism}. 
Then we conclude by Lemma \ref{vanish constructible}.

\end{proof}

\begin{lemma}\label{vanish constructible}
    Let $A$ be an animated ring and $f\in \pi_0(A).$ Then $M\in \Mod_A$ is $0$ iff $M/\!\!/f=0$ and $M[1/f]=0.$
\end{lemma}
\begin{proof}
     Note that $M/\!\!/f=0$ so we know that multiplication by $f$ on $M$ is an equivalence. In particular $M[1/f]\simeq M$ but the left hand side is $0.$ 
\end{proof}

\bibliographystyle{alpha}
\bibliography{references2}

\end{document}